\documentclass[hidelinks,onefignum,onetabnum]{siamart220329}
\usepackage{lipsum}
\usepackage{amsfonts}
\usepackage{graphicx}
\usepackage{epstopdf}
\usepackage{algorithmic}
\ifpdf
  \DeclareGraphicsExtensions{.eps,.pdf,.png,.jpg}
\else
  \DeclareGraphicsExtensions{.eps}
\fi

\newsiamremark{remark}{Remark}
\newsiamremark{hypothesis}{Hypothesis}
\crefname{hypothesis}{Hypothesis}{Hypotheses}
\newsiamthm{claim}{Claim}

\headers{Robins-Monro Augmented Lagrangian Method}{Rui Wang, Chao Ding}

\title{Robins-Monro Augmented Lagrangian Method for Stochastic Convex Optimization \thanks{This version: September 1, 2022.
\funding{This work was supported in part by the National Key R\&D Program of China 2021YFA1000300, 2021YFA1000301, the National Natural Science Foundation of China (No. 12071464), and the Beijing Natural Science Foundation (Z190002).}}}

\author{Rui Wang\thanks{
		(\email{wangrui2020@amss.ac.cn}, \email{dingchao@amss.ac.cn}).}
	\and Chao Ding\footnotemark[2]}

\author{Rui Wang\thanks{Institute of Applied Mathematics, Academy of Mathematics and Systems Science, Chinese Academy of Sciences, Beijing, P.R. China, School of Mathematical Sciences, University of Chinese Academy of Science, Beijing, P.R.
		China.
  (\email{wangrui2020@amss.ac.cn}).}
\and Chao Ding %\footnotemark[3]
\thanks{Institute of Applied Mathematics, Academy of Mathematics and Systems Science, Chinese Academy of Sciences, Beijing, P.R. China.
  (\email{dingchao@amss.ac.cn}).}
}

\usepackage{amsopn}

\usepackage{array} % <- Preamble
\usepackage[caption=false]{subfig} % <- Preamble
\newcolumntype{R}{>{$}r<{$}} %
\newcolumntype{V}[1]{>{[\;}*{#1}{R@{\;\;}}R<{\;]}} %

 \usepackage{graphicx,epstopdf} % <- Preamble
 \usepackage[caption=false]{subfig} 
\newtheorem{assumption}{Assumption}[section]
 
\ifpdf
\hypersetup{
  pdftitle={ROBINS-MONRO AUGMENTED LAGRANGIAN METHOD},
  pdfauthor={Rui Wang, Chao Ding}
}
\fi

\begin{document}

\maketitle

% REQUIRED
\begin{abstract}
In this paper,  we propose a Robbins-Monro augmented Lagrangian method (RM-ALM) to solve a class of constrained stochastic convex optimization, which can be regarded as a hybrid of the Robbins-Monro type stochastic approximation method and the augmented Lagrangian method of convex optimizations. Under mild conditions, we show that the proposed algorithm exhibits a linear convergence rate. Moreover, instead of verifying a computationally intractable stopping criteria, we show that the RMALM with the increasing subproblem iteration number has a global complexity $\mathcal{O}(1/\varepsilon^{1+q})$ for the $\varepsilon$-solution  (i.e., $\mathbb{E}\left(\|x^k-x^*\|^2\right) < \varepsilon$), where $q$ is any positive number. Numerical results on synthetic and real data demonstrate that the proposed algorithm outperforms the existing algorithms.
\end{abstract}

%Moreover, the complexity of $x$ is $\mathcal{O}(1/\varepsilon^{1+q})$, where $q$ is any postive number

% REQUIRED

% REQUIRED
{\small 
\section{Introduction}\label{intro}

In this paper, we consider the following stochastic convex optimization:
\begin{equation}\label{nlp}
	\begin{aligned}
		&\min _{x \in X }\  f(x) = f_0(x)+ f_1(x), \quad  f_1(x):=\mathbb{E}_{\xi}\left(F(x , \xi)\right) ,\\
		&\text { s.t. }\  h(x):=(h_{1}(x),\ldots,h_{M}(x))^T \leq 0,
	\end{aligned}
\end{equation}%$f_0$ is a deterministic function and $f_1(x) = \mathbb{E}_{\xi}\left[F(x ; \xi)\right]$, 
where $X \subseteq \mathbb{R}^{n}$ is a nonempty convex and compact set, $\xi$ denotes the random variable whose distribution $P$ is supported on sample space $\Omega$ and $\mathbb{E}_{\xi}$ is the expectation with respect to $\xi$, the continuously  differentiable functions $f_0: \mathbb{R}^n\mapsto \mathbb{R}$ and $h_j: \mathbb{R}^n \mapsto \mathbb{R}$, $j=1,\ldots,M$ are convex with respect to $x$, and the continuously  differentiable function $F: \mathbb{R}^n \times \Omega \mapsto \mathbb{R}$ is convex with respect to $x$ for {almost sure} $\xi \in \Omega$. In addition, we assume that the random variable $\xi$ is independent of $x$, which implies that {\cref{nlp} is a stochastic convex optimization. When $\xi$ is distributed uniformly on a finite set $\{\xi_{1},\ldots,\xi_{N}\}$, the problem \cref{nlp} reduces to the following convex optimization involving the finite-sum objective function:
	\begin{equation}\label{nlp2}
		\begin{aligned}
			&\min _{x \in X }\  f(x) = f_0(x)+ \frac{1}{N}\sum_{i=1}^{N}F(x , \xi_i) ,\\
			&\text { s.t. }\  h(x) \leq 0.
		\end{aligned}
	\end{equation}
	
	The stochastic convex optimization \cref{nlp} appears widely in a variety of applications, including the portfolio optimization \cite{rockafellar2000optimization}, the multi-stage stochastic optimization \cite{pflug2014multistage}, and constrained deep neural networks \cite{chen2018constraint}.  Below we give a few concrete examples and more applications of the stochastic convex optimization \cref{nlp} can be found from \cite{birge1997state,shapiro2003monte,shapiro2021lectures}.
	
	\textbf{Stochastic convex quadratically constrained quadratic program (QCQP)}. Consider the following stochastic convex QCQP:
	\begin{equation} \label{qcqpe}
		\begin{aligned}
			&\min _{{x} \in X}\ f(x)= \mathbb{E}\big(\frac{1}{2} \left\|\xi_{H} x-\xi_{c}\right\|^{2}\big),\\
			&\text { s.t. } h_j(x)=\frac{1}{2} x^{\top} Q_{j} x+a_{j}^{\top} x \leq b_{j},\ j=1, \ldots, M,
		\end{aligned}
	\end{equation} 
	where $\xi:=(\xi_H,\xi_{c})$ with $\xi_H\in \mathbb{R}^{p\times n}$ and $\xi_{c} \in \mathbb{R}^p$ are random variables, and the symmetric positive semidefinite matrices $Q_{j}\in \mathbb{R}^{n\times n}$, $a_j\in\mathbb{R}^n$, and $b_{j}\in\mathbb{R}$, $j=1,\ldots,M$ are deterministic. Clearly, this stochastic QCQP is of the form of \cref{nlp}.
	
	\textbf{Two-stage stochastic program}. Consider the following two-stage stochastic program:
	\begin{equation}\label{first}
		\min\limits_{x \in X}\  f\left(x\right)=f_{0}\left(x\right)+f_1\left(x\right) ,
	\end{equation}
	with $X \subset \mathbb{R}^{n}$ a nonempty convex and compact set, $f_0$ is a convex continuously differentiable function, and $f_1\left(x\right)=\mathbb{E}_{\xi}\left(F\left(x, \xi\right)\right)$ where  $\xi:=(\xi_f,\xi_g,\xi_A,\xi_B,\xi_b)$ with the random  matrices $\xi_A\in\mathbb{R}^{d\times n}$ and $\xi_B\in\mathbb{R}^{d\times m}$, and random vectors $\xi_f\in\mathbb{R}^p$, $\xi_g\in\mathbb{R}^q$ and $\xi_b\in \mathbb{R}^{d}$ summarizes all the random variables involved in the second stage: 
	\begin{equation}\label{second} 
		\begin{array}{rl}
			F\left(x , \xi\right):= \min\limits_{y\in Y} & f_{2}\left(x, y, \xi_f\right) \\
			\text{s.t.}&  \xi_A x+\xi_B y=\xi_b,\quad  g\left(x, y,\xi_g\right) \leq 0,
		\end{array}
	\end{equation}%In the problem above,  $\xi$ contains, in particular, the random elements in matrices $A, B$, and vector $b$. 
	where $Y \subset \mathbb{R}^{m}$ is a nonempty convex and compact set, $f_2$ and $g$ are continuously  differentiable functions and  jointly convex  with respect to the first stage decision variable $x$ and the second stage decision variable $y$.  One of the method for solving the two-stage stochastic program \eqref{first} is the sample average approximation (SAA) method \cite{kleywegt2002sample}, see \cite{shapiro2021lectures} for more details. By sampling $\xi_{1},\xi_{2},\ldots,\xi_{N}$ in the distribution to approximate  $\mathbb{E}_{\xi}\left(F\left(x, \xi\right)\right) $, we can rewrite the two-stage stochastic problem \eqref{first} as the following convex optimization involving the finite-sum objective function:  
	%When $\xi$ is independent of the constraints in the second stage of the problem, 
	\begin{equation}\label{single}
		\begin{aligned}
			&\min\limits_{x\in X,y_1,\ldots,y_N\in Y}\  f_{0}\left(x \right) + \frac{1}{N}\sum_{i=1}^{N} f_{2}\left(x, y_{i}, {\xi_f}_{i}\right)\\
			&\quad\quad\ \ \text{s.t.}\quad   {\xi_A}_i x+{\xi_B}_i y_i ={\xi_b}_i,\ g\left(x,y_i, {\xi_g}_i\right) \leq 0,\quad i=1,2,\ldots,N,
		\end{aligned}
	\end{equation}
	where $y_i$ represents the second stage decision corresponding to $\xi_i$. Thus, we may approximately solve the two-stage stochastic problem \eqref{first} by considering a stochastic convex optimization in the form of \cref{nlp2}.
	
	\textbf{Stochastic portfolio optimization}. The third motivating example is the portfolio optimization problem involving Conditional Value at Risk (CVaR). In a fundamental work \cite{rockafellar2000optimization}, Rockafellar and Uryasev show that a class of asset allocation problems can be modeled as:
	\begin{equation}\label{cvar}
		\mathrm{CVaR}=\min_{a,x\in X}\big\{a+\frac{1}{1-p} E\left([f(x,\xi)-a]_{+}\right)\big\},
	\end{equation}
	where $f$ is the loss associated with the decision vector $x$, to be chosen from a certain subset $X$ of $\mathbb{R}^n$, and the random vector $\xi$ in $\mathbb{R}^m$, $p\in(0,1)$ is a safety (reliability) level chosen by users, $a$ is a threshold of loss $f$. {When the return on a portfolio $x$ is the sum of the returns on the individual instruments in the portfolio, scaled by the proportions $x_i$. The loss is the negative of return and can be denoted as $f(x,\xi):= -(\xi_{1}x_1+\cdots+\xi_{n}x_n) = -\xi^Tx$.}
	%When $f(x,\xi) = -\xi^Tx $.
	Let $m:=E(\xi)$ be the average return for each asset assumed to be known (or estimated). We have $\mathbb{E}\left(\xi^Tx\right) \geq R \Rightarrow m^{T}x  \geq R$, where $R$ encodes a minimum desired return. The feasible set of portfolios can be written as 
	\begin{equation}\label{set}
		X=\big\{x \in \mathbb{R}^{n} \mid \sum_{i=1}^{n} x_{i}=1, x \geq 0,-m^{T} x \leq-R \big\} ,	
	\end{equation}% \cite{ruszczynski2003stochastic} 
	which is a nonempty convex and compact set. To minimize \cref{cvar} concerning $x$ and $a$, we approximate the expectation in \cref{cvar} by sampling $\xi_{1},\xi_{2},\ldots,\xi_{N}$ with respect to the distribution of $\xi$, then we obtain the approximate problem:
	\begin{equation}\label{cvar2}
		\min_{a,x\in X}\ a+\frac{1}{(1-p)N} \sum_{i=1}^{N}[-\xi_i^Tx-a]_{+}.
	\end{equation}
	Introducing auxiliary variables $y_i$ for $i =1,2,\ldots,N$, it is equivalent to minimizing the problem:%the following convex optimization involving the finite-sum objective function:
	\begin{equation} \label{cvar3}
		\begin{aligned}
			\min _{a, x \in X, y}\ &a+\frac{1}{(1-p)N} \sum_{i=1}^{N} y_{i}\\
			\mbox{s.t.}\quad  &y_{i} \geq -\xi_i^Tx-a, \quad y_{i} \geq 0, i=1, \ldots, N.
		\end{aligned}
	\end{equation}
	This problem can be reduced to \cref{nlp2}.
	
	The stochastic approximation (SA) is an efficient approach to solving the  {unconstrained stochastic convex optimization}. It is firstly proposed by Robbins and Monro \cite{robbins1951stochastic} in 1951 for  {strongly convex unconstrained stochastic optimization}. Recently, the SA-type algorithms have become popular even beyond the optimization community. This trend can be credited to some extent to the exciting developments in emerging fields such as machine learning  \cite{bottou2018optimization,pflug2012optimization,xiao2014proximal}. In terms of convergence analysis, for the stochastic convex  {optimization} with easy projection constraint,  Nemirovski et al. \cite{nemirovski2009robust} show that, under the assumption of strong convexity, the SA algorithm exhibits a rate of convergence $\mathcal{O}(1/\varepsilon)$, i.e. after $\mathcal{O}(1/\varepsilon) $ iterations, it holds that $\mathbb{E}\left(\|x^k-x^*\|^2\right) < \varepsilon$, where $x^k$ is the $k$-th iterate, and $x^*$ is the optimal point. 
	
	Recently, Lan and Zhou \cite{lan2020algorithms} extend the stochastic approximation idea to the {stochastic convex optimization} with the single deterministic/stochastic constraint  and propose the cooperative stochastic approximation (CSA) method. In \cite{lan2020algorithms},  Lan and Zhou show that the CSA method has the  $\mathcal{O}(1/\varepsilon^2)$ rate of convergence for both optimality gap and constraint violation, where $\varepsilon$ denotes the optimality gap and infeasibility. Moreover, when the objective function and constraint are both strongly convex, the rate of convergence of the CSA can be improved to $\mathcal{O}(1/\varepsilon)$. Then, Basu and Nandy \cite{basu2019optimal} extend it to the stochastic convex optimization with multiple constraints. Yu et al. \cite{yu2017online} propose an online algorithm for constrained stochastic convex optimization which achieves $\mathcal{O}(1/\sqrt{k})$ expected regret and constraint violation (see \cite[Thm 1 and 2]{yu2017online}). Moreover, Nemirovski et al. \cite{nemirovski2009robust} propose the saddle-point mirror stochastic approximation (MSA) to solve the convex-concave stochastic saddle-point problems. It includes the constrained stochastic convex optimization as a special case if certain constraint qualifications hold and achieves an ergodic convergence rate $\mathcal{O}(1/\sqrt{k})$ in terms of the primal-dual gap. The convergence analysis of the above methods mainly focuses on the objective function value and constraint violation. For the iterative sequence convergence property, the penalized stochastic gradient (PSG) method proposed by Xiao \cite{xiao2019penalized} owns  {the} convergence rates $\mathcal{O}(1/k^{1/4})$ about $\mathbb{E}(\|\bar{x}^k-x^*\|^2)$ under the restricted strong convexity assumption, where $\bar{x}^k$ is obtained by weighted average. For the general constrained stochastic convex optimization problem with the deterministic/stochastic constraints, Boob et al. \cite{Boob2021} propose a primal-dual proximal gradient based method, so-called constraint extrapolation (ConEx) method, in which the linear approximations of the constraint functions are used to define the extrapolation (or acceleration) step. Under the strong convexity assumption on the objective function, we know from \cite[Theorem 1]{Boob2021} that the ConEx method exhibits  a $\mathcal{O}(1/\varepsilon)$ rate of convergence in terms of $\mathbb{E}(\|x^k-x^*\|^2)$ for solving the stochastic convex optimization problem \eqref{nlp}. In this paper, we propose a Robbins-Monro augmented Lagrangian method for the stochastic convex optimization \cref{nlp}, which can be regarded as a hybrid of the stochastic approximation and traditional augmented Lagrangian method. Similar as the ConEx method proposed by \cite{Boob2021}, our augmented Lagrangian based algorithm exhibits a $\mathcal{O}(1/\varepsilon^{1+q})$ rate of convergence in terms of $\mathbb{E}(\|x^k-x^*\|^2)$ provided only the strong convexity of the objective function, where $q>0$ is an arbitrarily given number.  
	
	The classical augmented Lagrangian method (ALM) proposed by  Hestenes \cite{hestenes1969multiplier} and Powell \cite{powell1969method} as an algorithm for  {constrained optimization}. Recently, the ALM has been recognized as an efficient method for solving many optimization problems \cite{zhao2015newton,li2018highly,zhou2021semi}, due to its fast linear convergence rate \cite{rockafellar1976augmented,cui2017quadratic,cui2019r}. Furthermore, it also has some combinations for the constrained stochastic convex optimization. Zhang et al. \cite{2021Solving} propose a stochastic augmented Lagrangian-type algorithm named stochastic linearized proximal method of multipliers (SLPMM) for stochastic convex optimization, which achieves the complexity of $\mathcal{O}(1/\varepsilon^2)$ for objective reduction and constraint violation. The SLPMM builds small-scale constrained optimization problems by simultaneously sampling the objective function and constraints, then continuously solving the sampled small-scale optimization problem approximates the solution of the stochastic convex optimization. It is worth noting that the same technique has been applied in the other stochastic proximal point algorithm (SPPA) based algorithms. In fact, both Ryu and Boyd \cite{ryu2014stochastic}, and Milzarek et al. \cite{milzarek2022semismooth} propose to construct the small-scale problem for the sampled objective function and solve it sequentially by proximal point algorithm (PPA) to obtain the  {solution of original problem}. However, for the theoretical analysis of SLPMM, authors assume that the sampled subproblem can achieve the exact solution, which may not be satisfied in practice. Xu \cite{xu2020primal} also proposes an ALM-related method, the primal-dual stochastic gradient method, which alternatively decreases the primal and dual variables in the augmented Lagrangian function. It achieves the $\mathcal{O}(1/\sqrt{k})$ convergence rate of objective reduction and constraint violation for convex case and nearly $\mathcal{O}(\log(k)/k)$ rate for strongly convex case with $\mathbb{E}(\|x^k-x^*\|^2)=\mathcal{O}(\log(k)/k)$.
	
	Traditional inexact ALM controls the accuracy of subproblem solutions by suitable stopping criteria, resulting in the linear or asymptotic superlinear convergence rates \cite{rockafellar1976augmented}. However, verifying the stopping criteria for stochastic convex optimization is usually computationally expensive or even intractable due to the expectations. In order to solve \cref{nlp} and explore the convergence of inexact ALM without stopping criteria, we first analyze the convergence of the ALM with random perturbations called stochastic augmented Lagrangian method (SALM). Then, we propose a Robbins-Monro augmented Lagrangian method (RM-ALM) to solve the constrained stochastic convex optimization \cref{nlp} which is inspired by Robbins-Monro proximal point algorithm (RMPPA) \cite{toulis2021proximal}. In the RMPPA, Toulis et al. \cite{toulis2021proximal} study the unconstrained stochastic convex optimization by the PPA, which can be viewed as a dual method of the ALM \cite{rockafellar1976augmented}. It enlightens us to solve the constrained stochastic convex optimization by the ALM. However, the RMPPA uses the Robbins-Monro method to solve the PPA subproblem, and its convergence requires the  {subproblem iteration number} to converge to infinity consistently, which is usually  {unpractical} even for most applications. By the adequately chosen step size and  {subproblem iteration number}, we successfully overcome this shortage in the proposed RMALM.
	
	The RMALM can be viewed as a hybrid of the ALM and the Robbins–Monro type method. In the RMALM, we use Robbins–Monro type method to minimize the augmented Lagrangian function at each iteration of ALM, and each inexact solution can be viewed as a perturbation of the exact solution. Instead of verifying a stopping criteria, which is intractable computationally for stochastic optimizations, we show that the RMALM with the  {increasing subproblem iteration number} has a global complexity $\mathcal{O}(1/\varepsilon^{1+q})$ for the $\varepsilon$-solution  (i.e., $\mathbb{E}\left(\|x^k-x^*\|^2\right) < \varepsilon$), where $q$ is any positive number. The contributions of this paper exist in the following several aspects. 
	\begin{itemize}
		\item[1)] We obtain the almost sure convergence of the stochastic augmented Lagrangian method (SALM). By introducing random noise to the  ALM subproblem solution, the convergence properties of the results are analyzed from a probabilistic viewpoint.  
		
		\item[2)] We design a novel Robbins-Monro augmented Lagrangian method with a convergence rate arbitrarily close to $\mathcal{O}(1/\varepsilon)$. We use a stochastic algorithm to solve the ALM subproblem and avoid the verification of the stopping criteria, which is more practical. Under the strongly convex assumption of the objective function, we obtain the total iteration complexity $\mathcal{O}(1/\varepsilon^{1+q})$, where $q$ is any positive number. It is arbitrarily close to $\mathcal{O}(1/\varepsilon)$. The complexity is shown to be comparable with the existing related stochastic methods. %(In contrast to other algorithms, we discuss the convergence rate of $x$ without the strong convexity of constraints and solve the ALM subproblem inexactly to obtain the total iteration complexity. )
		 
		\item[3)] We show the practical performance of the proposed algorithm by testing it on the stochastic convex QCQP, a two-stage stochastic program, and a stochastic portfolio optimization problem. We compare the proposed method to the CSA method in \cite{lan2020algorithms}, the MSA method in \cite{nemirovski2009robust}, the PDSG-adp method in \cite{xu2020primal} and the APriD method in \cite{yan2022adaptive}. The numerical results demonstrate that RMALM can decrease the objective value, the constraint violation, and the parameter error faster than other methods in terms of iteration number and running time.
		
	\end{itemize}
	
	% The outline is not required, but we show an example here.
	The paper is organized as follows. Some preliminaries and the analysis of the SALM are in Section \ref{sec:pre}, the algorithm framework and convergence results are in Section \ref{sec:alg}, the numerical experiments are in Section \ref{test}, and the conclusions follow in Section \ref{conclusion}.

\section{SALM: Stochastic augmented Lagrangian method}
\label{sec:pre}
 Denote $\mathbb{R}_+^M:=\{z\in \mathbb{R}^M\mid z\geq0\}$ and $\mathbb{R}_{++}^M:=\{z\in \mathbb{R}^M\mid z>0\}$. For any given $x\in\mathbb{R}^n$, $y\in\mathbb{R}_+^M$ and $c>0$, the augmented Lagrangian function $L(x,y,c)$ of \eqref{nlp} takes the following form:
  {
 \begin{equation}\label{eq:def-ALF}
 	L(x, y, c):= f(x)  +\frac{c}{2}\big\|\big(h(x)+\frac{y}{c}\big)_+\big\|^2-\frac{1}{2c}\left\|y\right\|^2,
 \end{equation}
}
where $(\cdot)_+$ is the metric projection operator over the non-negative orthant. It is clear that the augmented Lagrangian function $L$ is convex about $x$.  Let $(x^0, y^0) {\in \mathbb{R}^n\times\mathbb{R}_+^M}$ be a given initial point.  For  {the} $k$-th iteration, the SALM for \eqref{nlp} is defined as follows 
 \begin{subequations}\label{alm}
 	 \begin{align} 
 		&\hat{x}^{k+1} = \arg \min\limits_{x\in X} L(x, y^{k}, c^{k}),\label{alma}\\
 		&x^{k+1}   =  \hat{x}^{k+1}  -   c^k \epsilon^{k+1}, \label{almb}  \\
 		&y^{k+1} = \max\left\{0,y^{k} + c^k h(x^{k+1})\right\},\label{almc} 
 	\end{align} 
 \end{subequations}
 where $h(x)=(h_1(x),\ldots,h_M(x))^T$ and $c^k>0$ is given.
%  stochastic errors to the solution of the ALM subproblem \eqref{alma} which is named the stochastic augmented Lagrangian method, and study the properties of the resulting procedure from a probabilistic perspective. 
% 
%Specifically, given an initial estimate $y_0$, we can calculate the ALM update $x_1^+$ and $y_1^+$ according to Equation \cref{alma} and \cref{alb}; next, we observe $x_1$ and $y_1$ with error $\epsilon_1$, which can be represented as $x_1 = x_1^+  - \epsilon_1$, $y_1  = \max\{0,y_{0} + c_0 h(x_{1})\} $; then, we can 
%compute the ALM update $x_2^+$ given $y_1$, and so on. This procedure is summarized below:
Let  $l:\mathbb{R}^n\times\mathbb{R}^M\to  {[-\infty,\infty]}$ be the ordinary Lagrangian function for \cref{nlp}, i.e.,
\begin{equation}\label{lagrange function}
	l(x,y) = \left\{
	\begin{array}{ll}
		 f(x)+ \left\langle h(x),y\right\rangle, &   \mbox{if $x\in X$ and $y\in \mathbb{R}_+^M$},\\
		 -\infty, &   \mbox{if $x\in X$ and $y\notin \mathbb{R}_+^M$}, \\
		\infty,  &   \mbox{if $x\notin X$}.
	\end{array}
	\right.
\end{equation}
Thus, for \cref{nlp}, the primal essential objective function $p:\mathbb{R}^n\to (-\infty,\infty]$ and the dual essential objective function $g:\mathbb{R}^M\to [-\infty,\infty)$ take the following forms, respectively: 
\begin{equation}\label{essential function}
	p(x) :=\displaystyle{\sup_{y\in \mathbb{R}^M}} l(x,y)  \quad \mbox{and} \quad g(y) := \inf\limits_{x\in \mathbb{R}^n} l(x,y).
\end{equation}
It is clear from \cite{rockafellar1970convex} that $l$ is a closed saddle-function. Since the convexity, continuity of $f$, $g$ and compactness of $X$, we know from \cite[Corollary 37.5.2]{rockafellar1970convex} that the mapping 
\begin{equation}\label{eq:def-Tl}
T_l:(x,y)\rightarrow\{(v,u)\mid(v,-u)\in\partial l(x,y)\}
\end{equation} 
is a maximal monotone operator in $\mathbb{R}^{n+M}$. Define the inverse of $T_l$ as a mapping $T_l^{-1}:(v,u)\rightarrow\{(x,y)\mid (v,u)\in T_l(x,y)\}$. The following definition on the Lipschitz continuity of the inverse of a maximal monotone operator is taken from \cite{rockafellar1976augmented}.

\begin{definition}\label{def:T-inv-Lip}
	For a maximal monotone operator $T$ from a finite dimensional linear vector space $\mathcal{Z}$ to itself, we say that its inverse $T^{-1}(w):=\{z\in\mathcal{Z}\mid w\in T(z) \}$, $w\in \mathcal{Z}$ is Lipschitz continuous at the origin with modulus $a \ge 0$  if there is a unique solution $\hat{z}$ to $z=T^{-1}(0)$, and for some $\tau>0$ we have $\|z-\hat{z}\| \leq a\|w\|$ whenever $ z \in T^{-1}(w)$ and $\|w\| \leq \tau$. 
\end{definition}

The dual essential objective function $g$ is a proper closed concave function on $\mathbb{R}^M$. Thus, the mapping $T_g:=-\partial g$ is a maximal monotone operator and the dual optimal solutions is given by $T_g^{-1}(0):=\{y\in\mathbb{R}^M\mid 0\in T_g(y)\}$. It is well-known (cf. e.g., \cite{rockafellar1976augmented}) that the exact ALM for \eqref{nlp} is equivalent with the following dual proximal point algorithm (PPA): for each $k$,
\begin{equation}\label{ppa}
	\hat{y}^{k+1}=P^{k}(y^k)=\left(I+c^{k} T_{g}\right)^{-1}(y^k)=\arg \max _{y \in \mathbb{R}^{M}}\left\{g(y)-\left(1 / 2 c^{k}\right)\left\|y-y^{k}\right\|^{2}\right\},
\end{equation} 
i.e.,
\begin{equation}\label{alb}
	\hat{y}^{k+1}:=\max\left\{0,y^{k} + c^k h(\hat{x}^{k+1})\right\}, 
\end{equation}
where $\hat{x}^{k+1}$ is given by \eqref{alma}.

{Next, we introduce some assumptions on the stochastic convex optimization \cref{nlp}, which will be used in the subsequent almost sure  convergence analysis of the SALM.}

%By the definition of $g$ and the $C^1$ property of $f$ and $h$, g is only non-differentiable at the edge of the domain. We give an alternative to the subdifferential of $g$.

%\begin{definition}\label{def}
%	Define ($y\in R_+^m$)
%	$$\tilde{\nabla} g(y)=\left\{
%	\begin{aligned}
%		&(\frac{\partial g(y)}{\partial y^{(1)}} ,\ldots,\frac{\partial g(y)}{\partial y^{(M)}})^T(y),\ \mbox{if $y> 0$},\\
%		&(\frac{\partial g(y)}{\partial y^{(1)}}, \ldots,\frac{\partial g(y)}{\partial y^{(M)}})^T( y^{(1)},\ldots,\mathop{0+}\limits_{y^{(j_1)}},\ldots,\mathop{0+}\limits_{y^{(j_{|S|})}},\ldots,y^{(M)}),\\
%		& \qquad \mbox{if exists $y^{(j_s)}=0$, $(j_s \in S =\{j_1,j_2,\ldots,j_{|S|}\})$},
%	\end{aligned}\right.
%	$$
%where $y^{(j)}$ represents the $j$-th dimension of the Lagrangian multiplier vector.

%\begin{remark}
%	Since $g$ is only non-differentiable at with zero vectors for $y\in R_+^m $, while right derivative exists. By \cref{def}, $\tilde{\nabla} g(y)$ is well difined for any $y\in R_+^m $. 
%\end{remark}

\begin{assumption}\label{ass3}
	The convex objective function $f$ of \cref{nlp} is strictly convex for all $x \in X$.
% f is $\mu-$strongly convex for all $x \in X$.
\end{assumption}

 The following assumption is natural and reasonable, since the set $X$ is compact and the function $h$ is continuous.
\begin{assumption}\label{ass2}
	The constraint function $h$ in \cref{nlp} is Lipschitz continuous over $X$ with parameter $L_h>0$, i.e.,
	\[
	\left\| h \left(x\right) - h \left(y\right)\right\| \leq L_{h}\left\|x-y\right\| \quad \forall\, x,y\in X.
	\]
\end{assumption}

%We have the following lemma regarding the convexity and strong convexity of augmented Lagrangian function.
%
%\begin{lemma}\label{strongly convex}  The  is convex about $x$. Furthermore, if we assume \cref{ass3}$(b)$ holds, the augmented lagrangian function $L(x,y,c)$ is also strongly convex about $x$ with modulus $\mu$.
%\end{lemma}
%\begin{proof}
%	As shown in \cref{alma}, 
%	since $h(x)+\frac{y}{c}$ and $(\cdot)_+$ is convex and the latter one is monotonous, we have $(h(x)+\frac{y}{c})_+$ is convex, and its range is $[0,+\infty)$. $(\cdot)^2$ is convex and monotonous on $[0,+\infty)$, hence, $||( h(x)+\frac{y}{c})_+||^2$ is convex. Together with the convexity of $f$, we have the augmented Lagrangian function $L(x,y,c)$ is convex about $x$. 
%	
%	With the strongly convexity of $f$, we obtain that $L(x,y, c)$ is strongly convex about $x$ with modulus $\mu$.
%\end{proof}

%Since the derivative of $g(y)$ does not exist when $y$ contains $0$ elements, we have the following definition as a substitute for the derivative.

%Therefore, we are able to define a generalized gradient $Jg$ of the dual essential objective function  $g$ on $ \mathbb{R}_+^M$ as follows.
%\begin{definition}
%	Under \cref{ass3}$(a)$, for any $y\in\mathbb{R}_+^M$,  define $Jg(y)$ by 
%	\begin{equation}\label{gy}
%	J g(y):=\left\{
%	\begin{array}{ll}
%	\nabla g(y) & \mbox{if $y\in \mathbb{R}_{++}^M$},\\
%	h(x(y)) & \mbox{if $y\in  \mathbb{R}_+^M\setminus \mathbb{R}_{++}^M$}.
%	\end{array}\right.
%	\end{equation}
%\end{definition}

It follows from (\ref{essential function}) that the dual essential objective function  $g$ is given by
\[
g(y)= \inf\limits_{x\in \mathbb{R}^n} l(x,y)  = \left\{ \begin{array}{ll}
	\inf\limits_{x\in X} \big\{ f(x) + \left\langle h(x),y\right\rangle\big\} & \mbox{if $y\in \mathbb{R}^M_+$,}\\ 
	-\infty &  \mbox{otherwise,}
\end{array} \right. \quad y\in\mathbb{R}^M.
\]  
It is clear that $l$ is continuously differentiable on $X\times \mathbb{R}_{+}^M$. Since $h_j$, $j=1,\ldots,M$ are convex and $X$ is compact, under \cref{ass3}, we know from the continuity of $l$ that for any $y\in \mathbb{R}_+^M$, $\inf\limits_{x\in X} l(x,y)$ has the unique solution $x(y)$. Moreover, we know from  \cite[Theorem 9]{1967Lagrange}, the optimal solution mapping $x(\cdot):\mathbb{R}_+^M\to \mathbb{R}^n$ is continuous. Therefore, it is clear that the corresponding dual essential objective function $g(y)= f(x(y))+\left\langle h(x(y)),y\right\rangle$, $y\in \mathbb{R}_+^M$ is also continuous. The following lemma is on the subdifferential \cite{rockafellar1970convex} of the concave function $g$.
\begin{lemma}\label{lemma1}
	For each $k$, let $\hat{y}^k$ be given by \eqref{alb}. Suppose \cref{ass3} holds. Then, we have that for each $k$,
	\begin{equation}\label{gradient}
	 x(\hat{y}^k)=\hat{x}^k\quad\mbox{and}\quad h(\hat{x}^k) \in \partial g(\hat{y}^k),
	\end{equation}
where $\hat{x}^k$ is given by \eqref{alma}.
\end{lemma}
\begin{proof}
Let $y\in \mathbb{R}_+^M$ be  {arbitrarily} given. If $y\in \mathbb{R}_{++}^M$,  then since $l$ is continuously differentiable on $X\times \mathbb{R}_{++}^M$, $x(y)$ is the unique solution of $\inf\limits_{x\in X} l(x,y)$ and $g$ is concave , it follows from the Danskin theorem \cite{danskin1966theory} (cf. e.g., \cite[Theorem 10.2.1]{facchinei2003finite}) and \cite[Theorem 25.2]{rockafellar1970convex} that for any $y\in \mathbb{R}_{++}^M$, $g$ is differentiable  at $y$ with the gradient  
\begin{equation}\label{gyy}
	\nabla g(y) = \nabla_y l(x(y),y) = h(x(y))\quad\mbox{and}\quad \nabla g(y)=\partial g(y).
\end{equation}
% Due to the continuity of $h$, we have $h(x(y))$ is continuous for $y\in \mathbb{R}_+^M $. % and $\nabla g(y)$ is continuous for $ y\in \mathbb{R}_{++}^M$
If  $y\in  \mathbb{R}_+^M\setminus \mathbb{R}_{++}^M$, then  denote $U$ as the set of indexes of $y$ whose elements are equal to $0$, i.e., $U:=\big\{j\in\{1,\ldots,M\}\mid y_j=0\big\}$. Define $\{\tilde{y}\}\in\mathbb{R}^M_{++}$ by
\[
\tilde{y}_j:=\left\{\begin{array}{ll}
	\delta_u & \mbox{if $j\in U$}, \\
	y_j & \mbox{otherwise},
\end{array}\right. \quad        {j\in\{1,\ldots,M\}, }  
\]
where $\delta_u>0$ for each $u\in U$. Therefore, by the continuity of $h(x(\cdot))$ over $\mathbb{R}_+^M $, we know that $\lim\limits_{\tilde{y}\rightarrow y} h(x(\tilde{y}))$ exists and equals to $h(x(y))$.
Thus, by combining \cref{gyy}, we obtain that 
\begin{equation*} 
	\lim\limits_{\tilde{y}\rightarrow y} \nabla g(\tilde{y}) =  \lim\limits_{\tilde{y}\rightarrow y} h(x(\tilde{y})) = h(x(y)) .	
\end{equation*}	
	Together with the concavity and the continuity of $g$ on $\mathbb{R}_+^M$, we obtain that for $\forall y'\in \mathbb{R}^M $
 \begin{equation}	\label{subdiff}
 \begin{aligned}
 \langle h(x(y)),y'-y \rangle = \lim\limits_{\tilde{y}\rightarrow y}\langle \nabla g(\tilde{y}),y'-\tilde{y} \rangle
 \geq  \lim\limits_{\tilde{y}\rightarrow y} g(y')-g(\tilde{y}) 
 = g(y')-g(y). 
 \end{aligned}
 \end{equation} 
 It follows from the definition of subdifferential that $h(x(y))\in \partial g(y)$.

For each $k$, from the process of ALM (\ref{alma}), since the augmented Lagrangian function $L$ is convex about $x$  {and $X$ is convex}, we have that 
\begin{equation}\label{kkt}
0 \in \nabla f(\hat{x}^k) + c^{k-1}\nabla h(\hat{x}^k) {^T} \big(h(\hat{x}^k)+ \frac{y^{k-1}}{c^{k-1}}\big)_++\mathcal{N}_X(\hat{x}^k),
\end{equation}
where $\mathcal{N}_X(\hat{x}^k)$ denotes the normal cone of $X$ at $\hat{x}^k$. Together with \cref{alb}, we obtain that
$$ 
0\in\nabla f(\hat{x}^k) + \nabla h(\hat{x}^k) {^T}\hat{y}^k+ \mathcal{N}_X(\hat{x}^k) ,
$$
which implies that under the convexity of $l$ about $x$,
$$ \hat{x}^k  = \arg\min\limits_{x\in X} f(x)+ \big\langle h(x),\hat{y}^k\big\rangle =\arg\min\limits_{x\in X} l(x,\hat{y}^k) ,$$ 
thus $ \hat{x}^k= x(\hat{y}^k)$ and $h(x(\hat{y}^k)) = h(\hat{x}^k)\in\partial g(\hat{y}^k)$,
which leads to \cref{gradient}.
\end{proof}

%After defining the gradient $Jg$ of $g$ for $y\in \mathbb{R}_+^M$, we make the following assumption about $Jg$ for the stochastic ALM.
%\begin{assumption}\label{assg}
%	Suppose that the $J g$ is Lipschitz continuous with parameter $L_g$ for all $y_{1}, y_{2} \in  \mathbb{R}_+^M$, i.e.
%		$
%		\left\|J g \left(y_{1}\right) - J g \left(y_{
%			2}\right)\right\| \leq L_g\left\|y_{1}-y_{2}\right\| .
%		$
%\end{assumption} 

%For the uniqueness of the solution to the dual problem of \cref{nlp} in \cref{stochastic}, we additionally assume that \cref{nlp} satisfies the strict Mangasarian-Fromovitz Constraint Qualification (SMFCQ).
%
%\begin{assumption}\label{ass4}
%	The SMFCQ is satisfied at $x^\ast$, where $x^\ast$ is optimal solution of \cref{nlp}, i.e. the family of vectors
%	$$
%	\nabla h_{j}(x^\ast), \quad \mbox{$j \in J$ is linearly independent},
%	$$
%	and there exists a $z \in \mathbb{R}^{n}$ such that
%	$$
%	\nabla h_{j}\left(x^\ast\right) z<0, \quad i \in K, \quad \nabla h_{j}\left(x^\ast\right) z=0, \quad j \in J, 
%	$$
%	where $ J = \{ j\mid y^\ast_j>0\}$, i.e. the constraints corresponding to multipliers greater than 0, and $K=\left\{j \mid y^\ast_j=0\right\}$ for multipliers are zeros, $y^\ast$ is the optimal solution of the dual problem of \cref{nlp}.
%\end{assumption}

In order to study the almost sure convergence of SALM \eqref{alm}, we make the following assumption on the random errors $\epsilon^{k}$.
\begin{assumption}\label{ass5}
	There exists a constant $\sigma>0$ such that for each $k$,
	$$
	\mathbb{E}\big(\epsilon^{k}\mid \mathcal{F}^{k-1} \big)=0 \quad \mbox{and} \quad \mathbb{E}\big(\big\|\epsilon^{k}\big\|^{2}\mid \mathcal{F}^{k-1}\big) \leq \sigma^{2},%\quad a.s.
	$$
	where $\mathcal{F}^{k-1}$ denotes the $\sigma$-algebra generated by $\{\epsilon^{1},\epsilon^{2},\ldots,\epsilon^{k-1}\}$.
\end{assumption}

% as the dual of stochastic proximal point algorithm in \cite{toulis2021proximal}
In the following theorem, inspired by \cite{toulis2021proximal}, we obtain the almost sure convergence of the SALM \cref{alm}, in which a supermartingale convergence property  \cite[Theorem 1]{robbins1971convergence} plays a crucial role.
\begin{theorem}\label{stochastic}
	Suppose $\cref{ass3}$, $\cref{ass2}$  {and} $\cref{ass5}$ hold. Let $c^{k} = c^{0} k^{-q}$ with $c^{0}>0$ and $q \in(\frac{1}{2},1]$. Then,  {if the dual problem of \cref{nlp} has the unique solution $y^\ast$,} the iterates $y^{k}$ of the SALM (\ref{alm})  {will} converge almost surely to $y^\ast$.
\end{theorem}
\begin{proof}
	By (\ref{alm}), we know that for each $k$,
\begin{eqnarray}
\|y^k - y^\ast\|^2 &=& \|(y^{k-1}+c^{k-1}h(x^k))_+ - y^\ast\|^2 \leq \|y^{k-1}+c^{k-1}h(x^k) - y^\ast\|^2 \nonumber\\
%& =& \|y^{k-1}- y^\ast\|^2+2c^{k-1}\big\langle y^{k-1}- y^\ast ,h(x^k) \big\rangle + {c^{k-1}}^2\|h(x^k) \|^2 \nonumber \\
& =&\|y^{k-1}- y^\ast\|^2 + 2c^{k-1}\big\langle y^{k-1}- y^\ast ,h(\hat{x}^k) \big\rangle  +  2c^{k-1}\big\langle y^{k-1}- y^\ast ,h(x^k)-h(\hat{x}^k) \big\rangle \nonumber\\
&& +{ (c^{k-1})}^2\|h(x^k) \|^2.\label{a1} 
\end{eqnarray}
By taking expectations in (\ref{a1}) conditional on $\mathcal{F}^{k-1}$, we obtain that 
\begin{equation}\label{main1}
\begin{aligned}
\mathbb{E}\big( 	\|y^k - y^\ast\|^2 |\mathcal{F}^{k-1}\big)\leq& \|y^{k-1}- y^\ast\|^2 + 2c^{k-1}\mathbb{E}( \big\langle y^{k-1}- y^\ast , h(\hat{x}^k) \big\rangle |\mathcal{F}^{k-1})\\
+  2c^{k-1}&\mathbb{E}\big( \big\langle y^{k-1}- y^\ast ,h(x^k)-h(\hat{x}^k) \big\rangle|\mathcal{F}^{k-1}\big) + (c^{k-1})^2\mathbb{E}\big( \|h(x^k)\|^2 |\mathcal{F}^{k-1}\big).
\end{aligned}
\end{equation}
For each $k$, denote  $R^k :=\mathbb{E}\left( \left\langle y^{k-1}- y^\ast , h(\hat{x}^k) \right\rangle |\mathcal{F}^{k-1}\right)$. Since $\hat{x}^k$ is $\mathcal{F}^{k-1}$-measurable, we have that for each $k$,
\begin{eqnarray}
R^k&=&\big\langle y^{k-1}- y^\ast ,h(\hat{x}^k) \big\rangle  = \big\langle y^{k-1}- \hat{y}^k ,h(\hat{x}^k) \big\rangle +\big\langle \hat{y}^k - y^\ast ,h(\hat{x}^k) \big\rangle \nonumber\\
&  =& \big\langle y^{k-1}- (y^{k-1} + c^{k-1}h(\hat{x}^k ))_+ ,h(\hat{x}^k) \big\rangle +\big\langle \hat{y}^k - y^\ast ,h(\hat{x}^k) \big\rangle \nonumber\\
& \leq& \big\langle \hat{y}^k - y^\ast ,h(\hat{x}^k) \big\rangle.\label{a2}
\end{eqnarray} 
%It follows from \cite[Proposition 1.1]{kyparisis1985uniqueness} that $y^\ast$ is unique dual solution under \cref{ass4}. Moreover, 
By the concavity of $g$, we know from  Lemma \ref{lemma1} that for each $k$,
$$\langle \hat{y}^k - y^\ast ,h(\hat{x}^k) \rangle\leq g(\hat{y}^k)-g(y^\ast) \leq 0,$$
which implies that $ R^k \leq 0 $.

On the other hand, by \cref{ass2}, \cref{ass5} and Jensen's inequality for conditional expectations (cf. e.g., \cite{chow2003probability}), we know that   %[7.1,Theorem 4]
\begin{eqnarray}
 &&\mathbb{E}\big( \big\langle y^{k-1}- y^\ast ,h(x^k)-h(\hat{x}^k) \big\rangle|\mathcal{F}^{k-1}\big) \leq \mathbb{E}\big( \| y^{k-1}- y^\ast\| \|h(x^k)-h(\hat{x}^k)\| |\mathcal{F}^{k-1}\big) \nonumber \\
 &\leq& (1+\| y^{k-1}- y^\ast\|^2) L_h\mathbb{E}\big( \|x^k-\hat{x}^k\| |\mathcal{F}^{k-1}\big) \nonumber \\
 &\leq& c^{k-1}(1+\| y^{k-1}- y^\ast\|^2) L_h\mathbb{E}\big( \| \epsilon^k\| |\mathcal{F}^{k-1}\big) \nonumber \\
 &\leq& c^{k-1}(1+\| y^{k-1}- y^\ast\|^2) L_h\sqrt{\mathbb{E}\big( \| \epsilon^k\|^2 |\mathcal{F}^{k-1}\big)} \nonumber \\
 &\leq& c^{k-1}(1+\| y^{k-1}- y^\ast\|^2) L_h\sigma , \quad a.s.\label{a3}
%\quad &=&\mathbb{E}\big( \big\langle y^{k-1}- y^\ast ,h(\hat{x}^k - \sqrt{c^{k-1}}\epsilon^k) -h(\hat{x}^k) \big\rangle|\mathcal{F}^{k-1}\big) \nonumber\\
%\quad &=&\mathbb{E}\big( \big\langle y^{k-1}- y^\ast , -\sqrt{c^{k-1}} \epsilon^k\nabla h (\hat{x}^k)+ c^{k-1}\mathcal{O}(\|\epsilon^k \|^2) 	 \big\rangle|\mathcal{F}^{k-1}\big) \nonumber\\
%\quad &=& \mathbb{E}\big( \big\langle y^{k-1}- y^\ast , -\sqrt{c^{k-1}} \epsilon^k\nabla h (\hat{x}^k)\big\rangle|\mathcal{F}^{k-1}\big)+ \mathbb{E}\big( \big\langle y^{k-1}- y^\ast ,  c^{k-1}\mathcal{O}(\|\epsilon^k \|^2) 	 \big\rangle|\mathcal{F}^{k-1}\big) .
\end{eqnarray}
%By \cref{ass5} and $\hat{x}^k$ is $\mathcal{F}^{k-1}$-measurable, we have that
%\begin{equation}\label{plus1}
%\mathbb{E}\big( \big\langle y^{k-1}- y^\ast , -\sqrt{c^{k-1}} \epsilon^k\nabla h (\hat{x}^k)\big\rangle|\mathcal{F}^{k-1}\big)=0,	
%\end{equation}
%and there exsits $K$ such that 
%\begin{eqnarray}
%\mathbb{E}\big( \big\langle y^{k-1}- y^\ast ,  c^{k-1}\mathcal{O}(\|\epsilon^k \|^2) 	 \big\rangle|\mathcal{F}^{k-1}\big)&\leq& 
%\mathbb{E}\big( c^{k-1} \| y^{k-1}- y^\ast\|  \mathcal{O}(\|\epsilon^k \|^2)  	|\mathcal{F}^{k-1}\big) \nonumber\\
%&\leq& c^{k-1} K\mathbb{E}\big(  \| y^{k-1}- y^\ast\| \|\epsilon^k \|^2 \|	|\mathcal{F}^{k-1}\big) \nonumber\\
%&\leq& c^{k-1}\sigma^2 K \sqrt{M} \| y^{k-1}- y^\ast\|^2.\label{plus2}
%\end{eqnarray}
%Substituting \cref{plus1} and (\ref{plus2}) into (\ref{a3}), we have
%\begin{equation}\label{aa3}
%\mathbb{E}\big( \big\langle y^{k-1}- y^\ast ,h(x^k)-h(\hat{x}^k) \big\rangle|\mathcal{F}^{k-1}\big)	\leq  c^{k-1}\sigma^2 K \sqrt{M} \| y^{k-1}- y^\ast\|^2.
%\end{equation}
	
	It follows from \cref{ass2} and the boundness of $x$ that for each $k$, 
	\begin{equation}\label{last}
		\begin{aligned}
			\|h(x^k)\|^2 &= \| h(x^k) -  h(x^*) + h(x^*)\|^2
			\leq 2\|h(x^k) -  h(x^*) \|^2 + 2\|  h(x^*) \|^2\\
			&\leq  2L_h^2d^2 +2\|  h(x^*) \|^2,
		\end{aligned}
	\end{equation}
where $d:=\max\limits_{\forall  x,x^\prime\in X} {\|x - x'\|}$. Then take expectation in \cref{last} conditional on $\mathcal{F}^{k-1}$, we conclude that for each $k$,
	\begin{equation}\label{a4}
		\mathbb{E}\big( \|h(x^k)\|^2 |\mathcal{F}^{k-1}\big) \leq 2L_h^2d^2 +2\|  h(x^*) \|^2. 
	\end{equation}
	It then follows from \cref{main1}, (\ref{a2}), (\ref{a3}) and (\ref{a4}) that 
	$$
	\mathbb{E}\big( 	\|y^k - y^\ast\|^2 |\mathcal{F}^{k-1}\big) \leq (1+(c^{k-1})^2A_1)\| y^{k-1} -  y^\ast \|^2 + 2 c^{k-1}R^k + (c^{k-1})^2A_2,\quad a.s.,
	$$
	where $A_1:=2L_h\sigma$ and $A_2 :=2L_h^2d^2 +2\|  h(x^*) \|^2+2L_h\sigma$ are constants. Since the random variable $R^{k}$ is non-positive, $\sum c^{k}=\infty$ and $\sum (c^{k})^{2}<\infty$, by employing \cite[Theorem 1]{robbins1971convergence}, we know that $\left\|y^{k}-y^\ast\right\|^{2}$ converges to some $B\ge 0$ and $\sum c^{k-1} R^{k}>-\infty$ almost surely.
	
	 If $B \neq 0$, we have $\lim \inf \left\|y^{k}-y^\ast\right\|>0$ almost surely. If $\lim \inf R^k= 0$ almost surely, by (\ref{a2}) and \cref{lemma1}, we have $\lim \inf \left\langle \hat{y}^k- y^\ast ,h(\hat{x}^k) \right\rangle= 0$ almost surely. Thus,  there exists a subsequence $\{(\hat{x}^{k_j},\hat{y}^{k_j})\}$ satisfying $$\lim  \left\langle \hat{y}^{k_j} - y^\ast ,h(\hat{x}^{k_j}) \right\rangle= 0,\quad \mbox{ almost surely}.$$ 
	Since  $\hat{y}^k$ satisfies $\hat{y}^k = P^k(y^{k-1})$ and $P^k(y^\ast) = y^\ast$ by \cref{ppa} for all $k$, it then follows from the nonexpansive of $P^k$ \cite{rockafellar1976monotone}, 
	we have that for each $k$,
	\begin{equation}\label{last2}
	\begin{aligned}
	\big\|\hat{y}^{k}-y^\ast\big\|^{2} &= \|P^k(y^{k-1}) - P^k(y^\ast)  \|^2\leq \|y^{k-1}-y^\ast \|^2.
	\end{aligned}
	\end{equation}
	By \cref{last2} and $\left\|y^{k}-y^\ast\right\|^{2}$ converges to some $B\ge 0$ almost surely, the subsequence $\{\hat{y}^{k_j}\}$ is bounded almost surely. Then, there exists a subsequence $\{(\hat{x}^{k_{j(i)}},\hat{y}^{k_{j(i)}})\}$ and a point $(\hat{x}^*,\hat{y}^*)$ satisfying that $\lim \hat{y}^{k_{j(i)}} = \hat{y}^*$ and $\lim \hat{x}^{k_{j(i)}} = x(\hat{y}^{k_{j(i)}}) = \hat{x}^*$ almost surely by Lemma \ref{lemma1}. It follows the continuity of $h$ that $ \left\langle \hat{y}^* - y^\ast , h(\hat{x}^*)\right\rangle= 0$. Since $y^\ast$ is the unique solution for the dual problem of \cref{nlp} and $g$ is concave, we have that $g$ is strictly concave at $y^\ast$, i.e. $ \left\langle \hat{y}^* - y^\ast ,  h(\hat{x}^*) \right\rangle< 0$ for $\hat{y}^*\neq y^\ast$, which implies $\hat{y}^*= y^\ast$. By \cref{ass2} and \cref{ass5}, we have that
	$$  
	\begin{aligned}
	\mathbb{E}\big(\|y^k - \hat{y}^k \|^2\big) &= \mathbb{E}\big(\mathbb{E}(\|y^k - \hat{y}^k\|^2\mid \mathcal{F}^{k-1})\big) \\
	&\leq \mathbb{E}\big(\mathbb{E}(\|(y^{k-1} + c^{k-1}h(x^k ))_+ - (y^{k-1} + c^{k-1}h(\hat{x}^k))_+ \|^2\mid \mathcal{F}^{k-1})\big)\\
	&\leq (c^{k-1})^2\mathbb{E}\big(\mathbb{E}(\|h(x^k ) -h(\hat{x}^k) \|^2\mid \mathcal{F}^{k-1})\big)\\
	&\leq (c^{k-1})^4L_h^2\mathbb{E}\big(\mathbb{E}(\|\epsilon^k \|^2\mid \mathcal{F}^{k-1})\big)\leq  (c^{k-1})^4L_h^2\sigma^2 \rightarrow 0.
	\end{aligned}
	$$
	Let $k=k_{j(i)}$, it follows from Lebesgue's dominated convergence theorem that 
	$$
	\lim \mathbb{E}\big(\|y^{k_{j(i)}} - \hat{y}^{k_{j(i)}} \|^2\big) = \mathbb{E}\big(\lim\|y^{k_{j(i)}} - \hat{y}^{k_{j(i)}} \|^2\big)=
	\mathbb{E}\big(\lim\|y^{k_{j(i)}} - y^\ast \|^2\big)=0.
	$$
	The above equation implies $\lim\|y^{k_{j(i)}} - y^\ast \|^2=0$ almost surely, which is contradictory to $\lim \inf \left\|y^{k}-y^\ast\right\|>0$ almost surely. Therefore, we know that $\lim \inf R^k< 0$ almost surely, which implies that the series $\sum c^{k-1} R^{k}$  diverges almost surely since $\sum c^{k}=\infty$. This is contradictory to $\sum c^{k-1} R^{k}>\infty$ almost surely. Thus, we have $B=0$. This completes the proof.
\end{proof}

\section{RMALM: Robbins-Monro augmented Lagrangian method}
\label{sec:alg} 

We know from \cref{stochastic} that under suitable conditions, the general SALM can still be guaranteed to converge almost surely to the original solution of \eqref{nlp}. In this section, we shall present a practical SALM, so-called the Robbins-Monro augmented Lagrangian method (RMALM) for the stochastic convex optimization \cref{nlp}, in which the Robbins-Monro type method \cite{robbins1951stochastic} is employed to solve the stochastic subproblem \eqref{alma}, inexactly. In addition, it is also worth noting that when $M$ is large, the computation of the function value and gradient of the augmented Lagrangian function \cref{eq:def-ALF} is expensive or even intractable. To cope with this difficulty, we may rewrite the following summation in the form of an expectation, i.e.,
$$
\begin{aligned}
	\frac{c^k}{2}  \|\big(h(x)+\frac{y^k}{c^k}\big)_+ \|^2-\frac{1}{2c^k} \|y^k \|^2 &= \sum_{j=1}^{M} \frac{c^k}{2}\big(h_j(x)+\frac{y^{k}_{j}}{c^k}\big)_+^2 - \frac{1}{2c^k} {y^{k}_{j}}^2 \\
	&= \frac{1}{M}\sum_{j=1}^{M} \tilde{h}(x,\zeta_j,y^k,c^k )=\mathbb{E}(\tilde{h}(x,\zeta,y^k,c^k)) ,
\end{aligned}
$$ 
where for each $j$, $ \tilde{h}(x,\zeta_j {,y^k,c^k} ) = M \big( \frac{c^k}{2}\big(h_j(x)+\frac{y^{k}_{j}}{c^k}\big)_+^2 - \frac{1}{2c^k} {y^{k}_{j}}^2\big)$ denotes the $j$-th item in this summation and $\zeta$ is a random variable, which follows a discrete uniform empirical density distribution $Z$ of $\{\zeta_1,...,\zeta_M\}$.  Since the independence of $\xi$ and $\zeta$, we may further rewrite the augmented Lagrangian function \cref{eq:def-ALF} in the expectation form, $L(x,y^k,c^k) = \mathbb{E}[\tilde{L}(x,y^k,c^k,\xi,\zeta)]$, where $ \tilde{L}(x,y^k,c^k,\xi,\zeta) = f_0(x)+F(x,\xi)+  \tilde{h}(x,\zeta {,y^k,c^k} ) $. Thus, the $k$-th iteration of SALM (\ref{alm}) takes the following inexact form:
\begin{subequations}\label{alm1}
	\begin{align} 
		&x^{k+1} \approx \arg \min\limits_{x\in X} \mathbb{E}(\tilde{L}(x,y^k,c^k,\xi,\zeta)),\label{alm1a}\\
		&y^{k+1} = \max\big\{0,y^{k} + c^k h(x^{k+1})\big\},\label{alm1c} 
	\end{align} 
\end{subequations}
where the subproblem \cref{alm1a} is solved by the  Robbins-Monro type method \cite{robbins1951stochastic}: for each $k$ and a given  integer $S^{k+1}>1$,
\begin{equation}\label{rmalm}
	\begin{aligned}
		& w_{1}^k =x^{k},  \\
		& w_{s+1}^k ={\rm prox}_X\{w_{s}^k-\gamma_{s}^k\nabla_{w}\tilde{L}(w_{s}^k, y^k,c^k,\xi_{s}^k,\zeta_{s}^k) \} , \quad s=1, \ldots, S^{k+1}-1,  \\
		& x^{k+1} =w_{S^{k+1}}^k, 
	\end{aligned}
\end{equation}
where $\gamma_{s}^k>0$ a given constant. Overall, the Robbins-Monro augmented Lagrangian method (RMALM) for solving the stochastic convex optimization \cref{nlp} is summarized as \cref{alg}. 
\begin{algorithm}
	\caption{Robbins-Monro Augmented Lagrangian Method}
	\label{alg}
	\begin{algorithmic} 
		\STATE{Initial $x^0\in \mathbb{R}^n$, $y^0\in \mathbb{R}^M_+$. Define $\gamma_s^k:=\tau_s\eta^k/(n+\beta)$, where $\beta>0$ and $\{\tau_s\}$ and $\{\eta^k\}$ are bounded sequences. $\{c^k\}$ is a given positive sequence.}
		\FOR{$k=0,1,2,\ldots,K-1$}
		\STATE{$w_1^k=x^k$.}
		\WHILE{$1\leq n\leq S^{k+1}-1$}
		\STATE{Random sample $\xi_{s}^k$ and $\zeta_{s}^k$.}
		\STATE{$w_{s+1}^k =\operatorname{prox}_X\{w_{s}^k-\gamma_{s}^k\nabla_{w}\tilde{L}(w_{s}^k, y^k,c^k,\xi_{s}^k,\zeta_{s}^k) \}$.}
		\ENDWHILE
		\STATE{$x^{k+1} =w_{S^{k+1}}^k$.}
		\STATE{$y^{k+1} = \max\left\{0,y^{k} + c^k h(x^{k+1})\right\}$.}
		\ENDFOR
		\RETURN $x^K$
	\end{algorithmic}
\end{algorithm}

Traditional inexact ALM stops the subproblem by suitable stopping criteria. However, due to expectations, verifying the stopping criteria for stochastic convex optimization is usually computationally intractable. In our method, the subproblem iterates a given number of steps $S^k$. In the next subsection, we shall study the convergence of the RMALM with the fixed $S^k\equiv S$ and the increasing $S^k$, respectively.

\subsection{Convergence analysis of RMALM}

Since there is no guarantee that $x^{k+1}$ is an unbiased estimate of the exact solution $\hat{x}^{k+1}$ of the augmented Lagrangian subproblem \eqref{alma}, in stead of studying the almost sure convergence of the iteration sequence $\{y^k\}$ as in the \cref{stochastic}, we shall study the convergence properties of the RMALM in the sense of expectation.  

	To conclude, we need the following assumption on the dual essential objective function $g$.
	
	\begin{assumption}\label{lipschitz}
		$g$ is a $\alpha-$strongly concave function.
	\end{assumption}
	
	The above assumption is natural, and we will give the following examples where \cref{lipschitz} may satisfy.
	\begin{itemize}
		\item[(1)] Linear constraint: when the constraint $h(x) = Ax-b \leq 0$, from \cite[Proposition 2.7]{guigues2021inexact} we have $g$ is strongly concave on $\mathbb{R}^M$ with strong concavity $\frac{\lambda_{\min}(AA^T)}{L_f}$, where $L_f$ is Lipschitz constant of $\nabla f$.
		
		\item[(2)] Quadratically constrained quadratic program: when the problem is formulated as follows
		\begin{equation}\label{low} 
			\left\{\begin{array}{ll}
				\min\limits_{x \in \mathbb{R}^{n}}&\ f(x):=\frac{1}{2} x^{T} Q_{0} x+a_{0}^{T} x+b_{0}, \\
				\text{s.t.}&\	h_{1}(x):=\frac{1}{2} x^{T} Q_{1} x+a_{1}^{T} x+b_{1} \leq 0.
			\end{array}\right.
		\end{equation}
		Assume that $Q_{0}, Q_{1}$ are positive definite, $a_{0} \neq Q_{0} Q_{1}^{-1} a_{1}$, and there exists $x_{0}$ such that $h_{1}\left(x_{0}\right)<0$. Let $l$ be any lower bound of the optimal value about \cref{low}, and 
		$$\bar{\mu}=\left(l-f\left(x_{0}\right)\right) / h_{1}\left(x_{0}\right) \geq 0.$$ Then from \cite[Proposition 2.8]{guigues2021inexact}, the dual essential objective function $g$ is strongly concave on the interval $[0, \bar{\mu}]$ with constant of strong concavity
		\begin{eqnarray*}
			\alpha_{D}&=&\big(Q_{1}^{-1 / 2}\big(a_{0}-  Q_{0} Q_{1}^{-1} a_{1}\big)\big)^{T}\big(Q_{1}^{-1 / 2} Q_{0} Q_{1}^{-1 / 2}+\bar{\mu} I_{n}\big)^{-3} Q_{1}^{-1 / 2}\left(a_{0}-Q_{0} Q_{1}^{-1} a_{1}\right) \\
			&>& 0.
		\end{eqnarray*}
		\item[(3)] General nonlinear constraint: for general nonlinear constraint, we usually need the linear independence of $\{\nabla h_j(x):j=1,2,\ldots M\}$ at optimal points by \cite[Theorem 2.10]{guigues2021inexact}, thus M cannot be larger than $n$ theoretically. In practical, tests in Section \ref{test} also show favorable experimental results on $n \ll M$.	
	\end{itemize}

	Recall that for each $k$, $\hat{x}^{k+1}$ and $\hat{y}^{k+1}$ are the exact solutions of \cref{alma} and \cref{ppa}. The following lemma is on the iteration error estimations  between $\{\hat{x}^{k+1},\hat{y}^{k+1}\}$ and the inexact solutions $\{{x}^{k+1},{y}^{k+1}\}$. 
	
	\begin{lemma}\label{exactalm}
		For each $k$, let $\hat{x}^{k+1}$ and $\hat{y}^{k+1}$ be the exact solutions of \cref{alma} and \cref{ppa}, respectively. Suppose Assumption \ref{lipschitz} holds. Then, we have for each $k$,
		\begin{equation}\label{y}
			\|\hat{y}^{k+1}-y^\ast\|^2\leq  \theta^k\|  {y}^k-y^\ast\|^2,
		\end{equation}
		where $\theta^k = \left(1+\alpha c^{k}\right)^{-2} <1$ and $y^\ast$ is the unique dual optimal solution of \cref{nlp}. 
		%Moreover, ${x^k}$ is also convergence.
		
		Let $T_l$ be the mapping defined by \eqref{eq:def-Tl}. If we further assume that the inverse $T_l^{-1}$ is Lipschitz continuous at the origin with modulus $a_l>0$, then for each $k$,
		\begin{equation} \label{x}
			{\|\hat{x}^{k+1}-x^\ast\|}^2\leq {\theta^k}'{\| {y}^k-y^\ast\|}^2,
		\end{equation}
		where ${\theta^k}' =  {\left[\frac{(2+\alpha c^k)a_l}{c^k+\alpha (c^k)^2}\right]}^2 $ and $x^\ast$ is the unique optimal solution to \cref{nlp}.
	\end{lemma}
	\begin{proof}
		It follows from Assumption \ref{lipschitz} and \cite[Exercise 12.59]{rockafellar2009variational} that $T_g=-\partial g$ is strongly monotone operator with modulus $\alpha$. It then follows from \cite[(1.15)]{rockafellar1976monotone} that there exist the unique solution $y^\ast$ satisfying $0\in T_g(y^\ast)$ and for each $k$,
		\begin{equation*}
			\|\hat{y}^{k+1}-y^\ast\|=\|P^{k}\big( {y}^k\big)-P^{k}\big(y^\ast\big)\| \leq\big(1+\alpha c^{k}\big)^{-1}\| {y}^k-y^\ast \|,
		\end{equation*}
		which implies (\ref{y}) holds.
		
		Since that $T_l^{-1}$ is Lipschitz continuous at origin with modulus $a_l>0$, we know from Definition \ref{def:T-inv-Lip} that the unique primal and dual solution $x^\ast$ and $y^\ast$ of \cref{nlp} satisfy $ T_l^{-1}(0)={(x^\ast,y^\ast)}$ and for each $k$,
		\begin{equation}\label{lip}
			\|(\hat{x}^{k+1}, \hat{y}^{k+1})-(x^\ast, y^\ast)\| \leq  a_{l} \operatorname{dist}((0,0),T_l(\hat{x}^{k+1}, \hat{y}^{k+1})).
		\end{equation}
		For each $k$, donate \begin{equation*}
			\Phi_k(x)= \left\{
			\begin{array}{ll}
				L(x,y^k,c^k),&\quad \mbox{if $x\in X$,}\\
				\infty,&\quad\mbox{otherwise.}
			\end{array}			
			\right.
		\end{equation*}
		Then, by \cite[(4.21)]{rockafellar1976augmented} we obtain that for each $k$,
		\begin{equation}\label{corollary}
			\operatorname{dist}((0,0),T_{l}(\hat{x}^{k+1}, \hat{y}^{k+1})) \leq \big(\operatorname{dist}^{2}(0, \partial \Phi_k(\hat{x}^{k+1})+({c^{k}})^{-2}\|\hat{y}^{k+1}- {y}^{k}\|^{2}\big)^{1 / 2}.
		\end{equation}
		Since $\hat{x}^{k+1}$ is the optimal solution of \cref{alma}, we have $0\in\partial\Phi_k(\hat{x}^{k+1}) $. Then, we have for each $k$,
		\begin{equation*}
			\operatorname{dist}\big((0,0),T_{l}(\hat{x}^{k+1}, \hat{y}^{k+1})\big) \leq  {c^{k}}^{-1}\|\hat{y}^{k+1}- {y}^k\|.
		\end{equation*}
		This, together with (\ref{lip}), yields that for each $k$,
		\begin{equation*}
			\|(\hat{x}^{k+1}, \hat{y}^{k+1})-(x^\ast, y^\ast)\| \leq \frac{a_{l}}{c^k} ||\hat{y}^{k+1}- {y}^k||.
		\end{equation*}
		It then follows from (\ref{y}) that for each $k$, 
		\begin{eqnarray*}
			\|(\hat{x}^{k+1}, \hat{y}^{k+1})-(x^\ast, y^\ast)\| &\leq& \frac{a_{l}}{c^k} (\|\hat{y}^{k+1}-y^\ast||+|| {y}^k-y^\ast\|)\le \frac{a_{l}}{c^k}(1+\sqrt{\theta^k})\| {y}^k-y^\ast\| \\ 
			&=&\frac{(2+\alpha c^k)a_l}{c^k+\alpha ({c^k})^2}\| {y}^k-y^\ast\|.
		\end{eqnarray*}	
		This completes the proof.
	\end{proof}

	Under the following assumption, we know that the augmented Lagrangian function $L$ defined by \eqref{eq:def-ALF} is also strongly convex for all $x\in X$ with modulus $\mu$.
	\begin{assumption}\label{sc}
		The convex objective function $f$ of \cref{nlp} is $\mu-$strongly convex for all $x \in X$.
		% f is $\mu-$strongly convex for all $x \in X$.
	\end{assumption}
	
	The following assumption on the stochastic gradients of augmented Lagrangian function $L$ is standard in the convergence analysis of the Robbins-Monro type method.
	\begin{assumption}\label{variance}
		Suppose that there exists a constant $\sigma>0$ such that for each $k$ and  $1 \le s \le S^{k+1}-1$,   $\nabla \tilde{L}(w, y^k, c^k,\xi_{s}^k,\zeta_{s}^k) $ in \cref{rmalm} satisfies
		\begin{equation}
			\begin{aligned}
				&\mathbb{E}( \nabla \tilde{L}(w, y^k, c^k,\xi_{s}^k,\zeta_{s}^k)  |\mathcal{H}_{s-1}^k ) = \nabla L(w, y^k, c^k),\\
				\ &\mathbb{E}(\|\nabla \tilde{L}(w, y^k, c^k,\xi_{s}^k,\zeta_{s}^k)\|^2|\mathcal{H}_{s-1}^k) \leq \sigma^2,		
			\end{aligned}\quad \forall\, w\in X,
		\end{equation}
		where $\mathcal{H}_{s-1}^k$ denotes the $\sigma$-algebra generated by $\{\xi_{1}^0,\zeta_1^0,\ldots,\xi_{S^1-1}^0,\zeta_{S^1-1}^0,\ldots,\xi_{s-1}^k,\zeta_{s-1}^k\}$.
	\end{assumption}
	
	The following lemma is on the non-asymptotic estimation on the distance between an inexact solution $x^{k+1}$ generated by the Robbins-Monro type method \eqref{rmalm} and the exact solution $\hat{x}^{k+1}$ of \cref{alm1a}.
	\begin{lemma}\label{rm}
		Suppose \cref{sc} and \cref{variance} hold. For each $k$ and $1\le s\le S^{k+1}-1$, let   $\gamma_s^k=\tau_s\eta^k/(n+\beta)$ with $0<\underline{\tau}< \tau_s < \bar{\tau}$, $1/(\mu\underline{\tau})<\eta^k <\eta $ and $\beta >2\mu\eta\bar{\tau}-1$. Then, for each $k$, the $k$-iteration $x^{k+1}$ generated by  \cref{rmalm} satisfies the following inequality
		\begin{equation}\label{l1}
			\mathbb{E}\big(\|x^{k+1}-\hat{x}^{k+1}\|^{2}\big) \leq \frac{v}{S^{k+1}+\beta}  
		\end{equation}
		where  
		\begin{equation} \label{v}
			v:= \max \left\{\frac{\eta^{2}\bar{\tau}^2 { \sigma}^{2}}{2 \mu \eta\underline{\tau}-1}, (\beta+1)d^{2}\right\} \quad \mbox{with} \quad d:=\max\limits_{\forall  x,x^\prime\in X} {\|x - x'\|}.
		\end{equation}
		%and $\mathcal{H}^{K-1}$ denotes the $\sigma$-algebra generated by $\{ \xi_{1}^0,\zeta_1^0,\ldots,\xi_{S^1-1}^0,\zeta_{S^1-1}^0,\ldots,\xi_{S^k-1}^{k-1},\zeta_{S^k-1}^{k-1}\}$
	\end{lemma}
	\begin{proof}
		For each $k$ and $1\le s\le S^{k+1}-1$, denote $ B_{s}^k:=\left\|w_{s}^k-\hat{w}^k\right\|^{2}$ and $b_{s}^k:=\mathbb{E}\left(B_{s}^k\right)= \mathbb{E}\left(\left\|w_{s}^k-\hat{w}^k\right\|^{2}\right)$,
		where $\hat{w}^k$ represents the optimal of \cref{alm1a}. By \cref{rmalm} and the non-expansion property of the proximal mapping and noting that $\operatorname{prox}_X(\hat{w}^k)=\hat{w}^k$, we have for each $k$ and $s$,
		\begin{eqnarray}
			B_{s+1}^k & = & \| \operatorname{prox}_X\{w_{s}^k-\gamma_{s}^k\nabla_{w}\tilde{L}(w_{s}^k, y^k,c^k,\xi_{s}^k,\zeta_{s}^k) \} -\hat{w}^k\|^{2} \nonumber\\ 
			&\leq& \|w_{s}^k-\gamma_{s}^k \nabla_{w}\tilde{L}(w_{s}^k,y^{k},c^k, \xi_{s}^k,\zeta_{s}^k )-\hat{w}^k\|^{2} \nonumber \\ 
			&=& B_{s}^k+ {\gamma_{s}^k}^{2}\|\nabla_{w}\tilde{L}(w_{s}^k,y^{k},c^k, \xi_{s}^k,\zeta_{s}^k )\|^{2}-2\gamma_{s}^k(w_{s}^k-\hat{w}^k)^{T}  \nabla_{w}\tilde{L}(w_{s}^k,y^{k},c^k, \xi_{s}^k,\zeta_{s}^k ).\label{l2}
		\end{eqnarray}
		Under Assumption \ref{variance}, since $w_s^k$ is $\mathcal{H}_{s-1}^{k}$-measurable, we obtain that for each $k$ and $s$,
		\begin{eqnarray}
			&&\mathbb{E}\big((w_{s}^k-\hat{w}^k)^{T} \nabla_{w}\tilde{L} (w_{s}^k,y^{k},c^k, \xi_{s}^k,\zeta_{s}^k )\big) \nonumber\\
			&=&\mathbb{E}\big(\mathbb{E}((w_{s}^k-\hat{w}^k)^{T} \nabla_{w}\tilde{L} (w_{s}^k,y^{k},c^k, \xi_{s}^k,\zeta_{s}^k )\mid\mathcal{H}_{s-1}^{k}) \big)\nonumber\\
			&=&	\mathbb{E}\big((w_{s}^k-\hat{w}^k)^{T}\mathbb{E}( \nabla_{w}\tilde{L} (w_{s}^k,y^{k},c^k, \xi_{s}^k,\zeta_{s}^k )\mid\mathcal{H}_{s-1}^{k}) \big)\nonumber\\
			&=&\mathbb{E}\big((w_{s}^k-\hat{w}^k)^{T}  \nabla_{w}L (w_{s}^k,y^{k},c^k)\big). \label{l3}
		\end{eqnarray} 
		It then follows from the $\mu$-strong convexity of $L(w, y^k , c^k) $ that the minimizer $\hat{w}^k$ is unique, and for each $k$,
		\begin{equation}\label{convex}
			(w-\hat{w}^k)^{T}\big(\nabla_w L(w, y^k,c^k)-\nabla_w L(\hat{w}^k, y^k,c^k) \big) \geq \mu\|w-\hat{w}^k\|^{2} \quad \forall\, w\in X.
		\end{equation}
		Since $\hat{w}^k$ is the optimal of \cref{alm1a}, we have $(w-\hat{w}^k)^{T}\nabla_w L(\hat{w}^k, y^k,c^k) \geq 0$ for any $w\in X$. Thus, it follows from (\ref{convex}) that $(w-\hat{w}^k)^{T} \nabla_w L(w, y^k,c^k)  \geq \mu\|w-\hat{w}^k\|^{2}$. Therefore, we have for each $k$ and $s$,
		\begin{equation}\label{strong}
			\mathbb{E}\big((w_{s}^k-\hat{w}^k)^{T} \nabla_w L(w_s^k, y^k, c^k)  \big) \geq \mu \mathbb{E}\big(\|w_{s}^k-\hat{w}^k\|^{2}\big)= \mu b_{s}^k .	
		\end{equation}
		Due to \cref{variance}, we have that for each $k$ and $s$,
		\begin{equation}\label{ll}
			\begin{aligned}
				\mathbb{E}\big(\|\nabla_{w}\tilde{L}(w_{s}^k,y^{k},c^k, \xi_{s}^k,\zeta_{s}^k )\|^{2}\big)  =\mathbb{E}\big( \mathbb{E}\big(\|\nabla_{w}\tilde{L}(w_{s}^k,y^{k},c^k, \xi_{s}^k,\zeta_{s}^k )\|^{2}\mid \mathcal{H}_{s-1}^k\big) \big)\leq \sigma^2.
			\end{aligned}
		\end{equation}
		By taking the expectation of both sides of $(\ref{l2})$,  we know from \eqref{l3}, \cref{strong} and \cref{ll} that for each $k$ and $s$,
		\begin{equation}\label{l5}
			b_{s+1}^k \leq (1-2 \mu \gamma_{s}^k)b_{s}^k+ (\gamma_{s}^k)^{2} { \sigma}^{2}.
		\end{equation}
		
		Next, we shall show the following inequality by induction
		\begin{equation}\label{key1}
			b_{s}^k \leq \frac{V(\eta^k)}{s+\beta},\quad  s=1,2,\ldots
		\end{equation}
		where $V(\eta^k)= \max \big\{\frac{(\eta^k)^{2}\bar{\tau}^2 { \sigma}^{2}}{2 \mu \eta^k\underline{\tau}-1},(\beta+1)d^{2} \big\}$. In fact, for $s=1$, it is clear from the definition of $V(\eta^k)$ that \eqref{key1} holds. Suppose \cref{key1} holds for $s \geq 1$. Denote $\hat{s}:=s+\beta $. By noting that $1-\displaystyle{\frac{ 2 \mu \eta^k\tau_s }{\hat{s}}}\geq 0 $ (since $\beta >2\mu\eta\bar{\tau}-1$) and $\hat{s}^2\geq (\hat{s}+1)(\hat{s}-1)$, we obtain from \cref{l5} that
		$$
		\begin{aligned}
			b_{s+1}^k&\leq(1-\frac{ 2 \mu \eta^k\tau_s }{\hat{s}})b_s^k+\frac{ ({\eta^k})^{2}\tau_s^2 { \sigma}^{2}}{ \hat{s}^2}  
			\leq(1-\frac{ 2 \mu \eta^k\tau_s }{\hat{s}}) \frac{V(\eta^k)}{\hat{s}}+\frac{ ({\eta^k})^{2}\tau_s^2  { \sigma}^{2}}{ \hat{s}^2} \\
			&\leq (\frac{ \hat{s}-2 \mu \eta^k\tau_s }{\hat{s}^2})V(\eta^k)+\frac{ ({\eta^k})^{2}\tau_s^2  { \sigma}^{2}}{ \hat{s}^2} 
			\leq (\frac{ \hat{s}-1 }{\hat{s}^2})V(\eta^k)- \frac{2 \mu \eta^k\tau_s-1}{\hat{s}^2}V(\eta^k)+\frac{ ({\eta^k})^{2}\tau_s^2  { \sigma}^{2}}{ \hat{s}^2}\\ 
			&\leq \frac{V(\eta^k)}{\hat{s}+1}. 
		\end{aligned}  
		$$
		Thus, we know the inequality \eqref{key1} holds.
		
		Let $\psi(\eta^k) = \frac{({\eta^k})^{2}\bar{\tau}^2 { \sigma}^{2}}{2 \mu \eta^k\underline{\tau}-1}$. By noting that $\psi'(\eta^k)>0$ for all $\eta^k> 1/\mu\underline{\tau}$, we have $\psi(\eta^k)<\psi(\eta)$ since $\eta^k<\eta$. Thus, we obtain that 
		$$
		b_{s}^k \leq \frac{v}{s+\beta}.
		$$
		By noting that $x^{k+1} = w_{S^{k+1}}^k $ and $\hat{x}^{k+1}=\hat{w}^k$, we then obtain \cref{l1}. This completes the proof.
	\end{proof}

	The following theorem is on the non-asymptotic convergence property of the sequence $\{(x^k,y^k)\}$ with a fixed  {subproblem iteration number}, i.e., for each $k$, $S^k\equiv S$, where $1<S$ is a given integer.
	\begin{theorem}\label{SGDALM}
		Let $\{(x^k,y^k)\}$ be the sequence generated by \cref{alg}. Suppose that \cref{ass2}, \cref{lipschitz}, \cref{sc} and \cref{variance} hold.  Let $T_l$ be defined by \eqref{eq:def-Tl}. Suppose that $T_l^{-1}$ is Lipschitz continuous at origin with modulus $a_l>0$. For each $k$ and $1\le s\le S^{k+1}-1$, let $\gamma_s^k=\tau_s\eta^k/(n+\beta)$ with $\underline{\tau}< \tau_s < \bar{\tau}$, $1/(\mu\underline{\tau})<\eta^k <\eta $ and $\beta >2\mu\eta\bar{\tau}-1$. Let $c>0$ be such that $\theta:=(1+\alpha c)^{-2} <\frac{1}{2}$ and $c^k \equiv c$ for each $k$. Denote $\rho:=2\theta<1$ and ${\theta}':=  {\big[\frac{(2+\alpha c)a_l}{c+\alpha {c}^2}\big]}^2 $. Then, we have for each $k$,
		\begin{equation} \label{e1}
			\mathbb{E}\big(\|y^{k}-y^\ast\|^2\big) \leq   \frac{2c^2 L_h^2v}{S^k}  +2\theta \mathbb{E}\big(\|y^{k-1}-y^\ast\|^2\big)
		\end{equation}and
		\begin{equation} \label{e2}
			\mathbb{E}\big(\|x^{k}-x^\ast\|^2\big)\leq  \frac{2v}{S^k} + 2\theta'\mathbb{E}\big(\|y^{k-1}-y^\ast\|^2\big),
		\end{equation}
		where $v$ is the constant given by \eqref{v}.
		
		If we further assume $S^k \equiv S$ for all $k$ with a given integer $S>1$, then the following inequalities hold for each $k$,
		\begin{equation} \label{salm}
			\mathbb{E}\big(\|y^{k}-y^\ast\|^2\big)\leq  \frac{2c^2L_h^2v}{(1-\rho)S}  +\rho^{k} \|y^{0}-y^\ast\|^2
		\end{equation}and
		\begin{equation} \label{salm2}
			\mathbb{E}\big(\|x^{k}-x^\ast\|^2\big) \leq  \frac{2v}{S} +\frac{4c^2 L_h^2\theta'v}{(1-\rho)S } +2\theta'\rho^{k-1}  \|y^{0}-y^\ast\|^2.
		\end{equation}
	\end{theorem}
	\begin{proof}
		It follows from \cref{exactalm}, \cref{rm} and \cref{ass2} that for each $k$,
		$$
		\begin{aligned}
			\mathbb{E}\big(\|y^{k}-y^\ast\|^2\big) & \leq 2\mathbb{E}\big(\|y^{k}-\hat{y}^{k}\|^2\big)+2\mathbb{E}\big(\|\hat{y}^{k}-y^\ast\|^2\big) \\
			&\leq 2\mathbb{E}\big(\|y^{k}-\hat{y}^{k}\|^2\big)+ 2\theta \mathbb{E}\big(\|y^{k-1}-y^\ast\|^2\big) \\%\quad(\text { by \cref{exactalm}  } ) \\
			& \leq 2\mathbb{E}\big(\|y^{k-1}+ch(x^{k})-y^{k-1}- ch(\hat{x}^{k}) \|^2\big)+2\theta \mathbb{E}\big(\|y^{k-1}-y^\ast\|^2\big) \\%\quad(\text{by \cref{exp2}})\\
			& \leq 2c^2L_h^2\mathbb{E} \big(\|x^{k}-\hat{x}^{k}\|^2\big)+2\theta \mathbb{E}\big(\|y^{k-1}-y^\ast\|^2\big) \\%\quad  (\text{ by \cref{ass2}})\\
			& \leq 2c^2 L_h^2 \frac{v}{S^k}  +2\theta \mathbb{E}\big(\|y^{k-1}-y^\ast\|^2\big), %\quad(\text { by  \cref{rm}})\\
		\end{aligned}
		$$
		which implies \cref{e1} holds.
		Using \cref{exactalm} and \cref{rm} again, we obtain that for each $k$,
		\begin{equation*} 
			\begin{aligned}
				\mathbb{E}\big(\|x^{k}-x^\ast\|^2\big) & \leq 2\mathbb{E}\big(\|x^{k}-\hat{x}^{k}\|^2\big)+2\mathbb{E}\big(\|\hat{x}^{k}-x^\ast\|^2\big) \quad \\%\text { [triangle inequality] }  
				& \leq 2\mathbb{E}\big(\|x^{k}-\hat{x}^{k}\|^2\big)+ 2\theta' \mathbb{E}\big(\|y^{k-1}-y^\ast\|^2\big) \\%\quad(\text { by \cref{exactalm} })\\
				& \leq  \frac{2v}{S^k} + 2\theta'\mathbb{E}\big(\|y^{k-1}-y^\ast\|^2\big), %\quad(\text { by \cref{rm}})\\
			\end{aligned}
		\end{equation*}
		which implies \eqref{e2} holds. In addition, if  $S^k \equiv S$ for all $k$, since $\rho=2\theta=2(1+\alpha c)^{-2} <1 $, we obtain from the recursion of \cref{e1} that for each $k$,
		$$
		\mathbb{E}\big(\|y^{k}-y^\ast\|^2\big) \leq  \frac{2c^2L_h^2}{1-\rho} \frac{v}{S}  +\rho^{k} \|y^{0}-y^\ast\|^2.
		$$	
		Thus, we know that \cref{salm} holds. It then follows from the recursion of \cref{e2} and \cref{salm} that for each $k$,
		\begin{equation*} 
			\mathbb{E}\big(\|x^{k}-x^\ast\|^2\big)   \leq  \frac{2v}{S} +\frac{4c^2 L_h^2\theta'}{1-\rho } \frac{v}{S}+2\theta' \rho^{k-1} \|y^{0}-y^\ast\|^2 .
		\end{equation*}
		This completes the proof.
	\end{proof}

	From \cref{SGDALM}, we know that the RMALM (\cref{alg}) converges with the fixed  {subproblem} iteration number $S^k\equiv S$. Roughly speaking, the equation \cref{salm} and \cref{salm2} imply that $\mathbb{E}(\|y^{k}-y^\ast\|^2)$ and $\mathbb{E}(\|x^{k}-x^\ast\|^2)$ converge at the rate of $\mathcal{O}(1/S)$ in terms of $S$. Therefore, the $\varepsilon$-solution (i.e., $\mathbb{E}(\|x^k-x^\ast\|^2) < \varepsilon$, $\mathbb{E}(\|y^k-y^\ast\|^2) < \varepsilon$) can be obtained by choosing $k =\mathcal{O}(\log\frac{1}{\varepsilon} )$ and $S = \mathcal{O}(\frac{1}{\varepsilon})$. Overall,  the total iteration number to obtain a $\varepsilon$-solution is $\mathcal{O}(\frac{1}{\varepsilon}\log\frac{1}{\varepsilon})$. On the other hand, the convergence results obtained in \cref{SGDALM} imply that the  {subproblem} iteration number $S$ must converge to infinity, which is unpractical. However, we do not need $S^k$ to be very large at the beginning of the ALM algorithm. In intuition, as the algorithm proceeds, the subproblem of the ALM will require higher accuracy, which means more iterations. Therefore, we can set $S^k$ to get progressively larger, and obtain a practical complexity by choosing a more suitable  {subproblem} iteration number $S^k$. Below we present the convergence results of RMALM with the increasing $S^k$.
	
	\begin{theorem}\label{diminish}
		Suppose the conditions in \cref{SGDALM} hold. Let $S^k = \lceil S^0\rho^{-k(1+q)}\rceil $,  where $S^0>1$, $q>0$ are given constants and $\lceil a\rceil$ denotes the smallest integer larger than $a$ for any number $a$. Let $\{(x^k,y^k)\}$ be the sequence generated by \cref{alg}. Then we have the linear convergence rate of $\{y^k\}$
		\begin{equation} \label{rate}
			\mathbb{E}\big(\|y^{k}-y^\ast\|^2\big) = D\rho^k,
		\end{equation}
		where $D := \displaystyle{\frac{2c^2 L_h^2v\rho^{q}}{S^{0}(1-\rho^{q})} \frac{}{}+\|y_{0}-y^\ast\|^2}$.
		
		The $\varepsilon$-solution for $x^{k}$ and $y^{k}$, that satisfies $\mathbb{E}(\|x^{k}-x^\ast\|^2) <  \varepsilon$ and $\mathbb{E}(\|y^{k}-y^\ast\|^2) <  \varepsilon$ respectively, need $\mathcal{O}((\frac{1}{\varepsilon})^{1+q})$ iterations both.
	\end{theorem}
	\begin{proof}
		It follows from the recursion of \cref{e1} that for each $k$,
		\begin{equation*}
			\mathbb{E}\big(\|y^{k}-y^\ast\|^2\big) \leq  2c^2  L_{h}^2 v \big(\frac{1}{ S^{k} }+\frac{\rho}{ S^{k-1} }+\cdots+\frac{\rho^{k-1}}{ S^1 }\big)+\rho^{k} \|y_{0}-y^\ast\|^2.
		\end{equation*}
		Let $S^k = \lceil S^0\rho^{-k(1+q)}\rceil$, where $S^0>1$ and $q>0$ are given constants. We have
		\begin{equation*}
			\begin{aligned}
				\mathbb{E}\big(\|y^{k}-y^\ast\|^2 \big)&\leq \big(2c^2 L_h^2 \frac{v}{S^{0}} (  \rho^{kq} +\cdots+\rho^q) +\|y_{0}-y^\ast\|^2\big)\rho^{k} \\
				& \leq \big(2c^2 L_h^2  \frac{v}{S^{0}} \frac{\rho^{q}}{1-\rho^{q}}+\|y_{0}-y^\ast\|^2\big)\rho^{k},
			\end{aligned}
		\end{equation*}
		which implies that the $y^k$ has a linear convergence rate \cref{rate}.
		
		Under the requirement $\mathbb{E}(\|y^{k}-y^\ast\|^2) \leq \varepsilon$, we have 
		$		  k \geq (\ln\rho)^{-1}\ln \frac{\varepsilon}{D}.
		$
		Sum up the total iterations for $K= (\ln\rho)^{-1}\ln \frac{\varepsilon}{D}$, we obtain that 
		\begin{equation}\label{dd}
			\begin{aligned}
				\sum_{k=1}^{K} S^{k} =&\sum_{k=1}^{K} \lceil S^0\rho^{-k(1+q)}\rceil 
				\leq S^{0} \sum_{k=1}^{K}[\rho^{-(1+q)}]^{k}+K 
				= S^{0} \frac{\rho^{-(1+q)}(1-\rho^{-(1+q) \cdot K})}{1-\rho^{-(1+q)}} +K \\
				=&\frac{S^{0} \cdot \rho^{-(1+q)}}{\rho^{-(1+q)}-1} \rho^{-(1+q) \cdot K}-\frac{S^{0}\rho^{-(1+q)}}{\rho^{-(1+q)}-1}+K  =\mathcal{O}(\rho^{-(1+q) K}) +\mathcal{O}(K)\\
				=&\mathcal{O}(\rho^{-(1+q) \cdot (\ln\rho)^{-1}\ln \frac{\varepsilon}{D}}  ) +\mathcal{O}(\ln\frac{1}{\epsilon} )
				=\mathcal{O}\big((\frac{1}{\varepsilon})^{1+q}\big),
			\end{aligned}
		\end{equation}
		which implies the $\mathcal{O}((\frac{1}{\varepsilon})^{1+q})$ iteration complexity of $\mathbb{E}(\|y^{k-1}-y^\ast\|^2)$.
		
		It follows from \cref{e2} that
		\begin{equation*}
			\begin{aligned}
				\mathbb{E}(\|x^{k}-x^\ast\|^2)&\leq  \frac{2v}{S^{k}} + 2\theta'\mathbb{E}(\|y^{k-1}-y^\ast\|^2)
				\leq \frac{2v}{S^0}\rho^{k(1+q)} + 2\theta'D\rho^{k-1}\leq 2(\frac{v}{S^0}+\frac{\theta'D}{\rho})\rho^{k}.
			\end{aligned}
		\end{equation*}		
		Denote $D_0:=2(\frac{v}{S^0}+\frac{\theta'D}{\rho})$. Under the requirement $\mathbb{E}(\|x^{k}-x^\ast\|^2) \leq \varepsilon$, we have 			$		  k \geq (\ln\rho)^{-1}\ln \frac{\varepsilon}{D_0}.$ Similar to \cref{dd}, sum up the total iterations for $K= (\ln\rho)^{-1}\ln \frac{\varepsilon}{D_0}$, we also obtain that 
		\begin{equation*}
			\sum_{k=1}^{K} S^{k} = \sum_{k=1}^{K} \lceil S^{0} \rho^{- k(1+q)} \rceil
			=\mathcal{O}(\rho^{-(1+q) \cdot (\ln\rho)^{-1}\ln \frac{\varepsilon}{D_0}})+\mathcal{O}(\ln\frac{1}{\epsilon} )  	=\mathcal{O}\big((\frac{1}{\varepsilon})^{1+q}\big),
		\end{equation*}
		which implies the $\mathcal{O}((\frac{1}{\varepsilon})^{1+q})$ iteration complexity of $\mathbb{E}(\|x^{k-1}-x^\ast\|^2)$.
	\end{proof}

	Two results are stated in \cref{diminish}. Firstly, by setting {$S^k =  \lceil S^0\rho^{-k(1+q)} \rceil$}, where $S^0>1$ and $q>0$ are given constants, we can guarantee the linear convergence rate of the RMALM without the stopping criteria. {Without verifying the stopping criteria, the whole algorithm is much simpler and practical.} This idea can also be used for other subproblem-solving algorithms. Moreover, we obtain the total complexity of $	\mathbb{E}(\|x^{k}-x^\ast\|^2)$ is arbitrarily close to $\mathcal{O}(1/\varepsilon)$.

\section{Numerical experiments}
\label{test}

This section tests the proposed method (RMALM) on the stochastic convex QCQP, a two-stage stochastic program, and a stochastic portfolio optimization problem. We compare  {our method} to four existing methods, the CSA method in \cite{lan2020algorithms}, the MSA in \cite{nemirovski2009robust}, the PDSG-adp method in \cite{xu2020primal} and the APriD method in \cite{yan2022adaptive}. In \cite{lan2020algorithms}, the output of CSA is the weighted average of ${x}^{t}$ over $t \in \mathcal{B}^{k}=\{t = 1,2,\ldots  {k} \mid \widehat{G}_{t} \leq \eta_{t}\}$. Note that $\mathcal{B}^{k}$ may be empty for a small $k$. Therefore, we also compute the weighted average of ${x}^{t}$ overall $t=1,2,\ldots  {k}$ as in \cite{yan2022adaptive} and name the results of CSA as CSA1, CSA2 respectively. The parameters in the stochastic convex QCQP are the same as in the experiments of \cite{yan2022adaptive}. Specifically, we take $s=1$ and $J_g = 100$ in CSA; $\beta_{1}=0.9, \beta_{2}=$ $0.99$ in APriD.  In our algorithm, we set $S^k = \lceil5\times1.7^{k(1+0.0001)}\rceil$. Other parameters are different in each experiment, and we take the optimal parameters according to the experiment for comparison. 

In all experiments, our comparisons  {contain} the objective value $f(x^k)$, the averaged constraint violation measured by $\frac{1}{M} \sum_{j=1}^{M}\left[h_{j}({x}^k)\right]_{+}$, the maximum constraint violation measured by $\max _{j \in\{1,\ldots,M\}}\left[h_{j}(x^k)\right]_{+}$, the iteration error measured by $\|x^k-x_{opt} \|^2$ and the averaged error $\|\bar{x}^k-x_{opt}\|^2$ which has the better performance for other algorithms, where $x_{opt}$ is the optimal solution in every experiment and $\bar{x}^k$ is  {a kind of average of the history iteration points according to the algorithms}. All the tests are performed in MATLAB R2021b installed on Linux with Intel Xeon(R) Gold 6230R CPU @ 2.10GHz.

\subsection{QCQP with expectation objective}
\label{qcqptest}
In this subsection, we test the algorithms on the stochastic convex QCQP in the following form:
\begin{equation} \label{qcqp}
	\begin{aligned}
		&\min _{{x} \in X}\ f(x)= \mathbb{E}\big(\frac{1}{2} \left\|\xi_H x-\xi_{c}\right\|^{2}\big),\\
		&\text { s.t. } h_j(x)=\frac{1}{2} x^{\top} Q_{j} x+a_{j}^{\top} x \leq b_{j},\ j=1, \ldots, M.
	\end{aligned}
\end{equation} 
Here $X=[-10,10]^{n}$, $\xi_H\in \mathbb{R}^{p\times n}$ and $\xi_{c} \in \mathbb{R}^p$ are randomly generated, and their components are generated by standard Gaussian distribution and then normalized. For each $j \in\{1,\ldots,M\}$, $Q_{j}\in \mathbb{R}^{n\times n}$ is a randomly generated symmetric positive semidefinite matrix with unit 2-norm; $a_j$ is randomly generated independently by the standard Gaussian distribution and then normalized; $b_{j}$ is generated from the uniform distribution on $[0.1,1.1]$. 

 {In the experiment, we test on  QCQP instances of size $(n, p) = (10, 5)$ and $(200, 150)$ and $M=5$ and $10000$, respectively. In both instances, we set batch size is 50 and run $5 \times 10^{4}$ iterations. For the instances with small data size, we solve the approximation problem for the generated $10^5$ samples by using CVX\cite{gb08,cvx} to obtain the optimal solution $x_{opt}$. When running the code, we obtain the unbiased estimate of the gradient and function values by sampling the random variable $\xi_H,\xi_c$ over the above distribution.} For the instances with large data sizes, we employ an estimated optimal solution $x_{opt}$ which has the smallest objective value in the feasible set among all iterations of RMALM, APriD, CSA, MSA, and PDSG-adp. 

In \cref{figexp}, we report the objective value, the averaged constraint violation, the maximum constraint violation, the last iteration error, and the averaged error by iteration and time. In the first four columns of the \cref{figexp}, the results of the other four algorithms are about $\bar{x}^k$, which have convergence guarantees. However, the results of our algorithm are only about the current iteration point $x^k$. The distance from the current iteration point $x^k$ to the optimal point for all algorithms is shown in the last column , and we can see that the other four algorithms start to oscillate at a certain distance from the optimal point and fail to converge to the optimal point, while our algorithm converges steadily to the optimal point. The running time for the algorithms in the different settings is shown as \cref{timeexp}. All results show that the RMALM outperforms the other methods in different stochastic convex QCQP instances.

\begin{table}[tbhp]\label{timeexp}
	\footnotesize
	\captionsetup{position=top} %<- Needed for using subtables created with the subfig package
	\caption{ Running time (in seconds) for QCQP \cref{qcqp}.}
	\begin{center} 
		\begin{tabular}{|c|c|c|c|c|} \hline
			$(n,p)$  & \multicolumn{2}{c|}{$(10,5)$} & \multicolumn{2}{c|}{$(200,150)$} \\ \hline
			number of constraints & {$M=5$} & {$M=10000$}  & {$M=5$} & {$M=10000$}\\ \hline
			RMALM & {\bf 31.7} & {\bf 40.2}  & {\bf 9449.4} & {\bf 10798.9}   \\ \hline
			ApriD &  41.7 & 51.6 & 10829.0& 11724.5 \\ \hline
			MSA & 41.6 & 47.1    & 10661.5 & 11563.9  \\ \hline
			CSA & 40.8 & 53.2     & 10661.1 & 12453.6 \\ \hline
			PDSG\_adp & 42.8 & 44.2 & 10530.5 & 11010.2\\ \hline
		\end{tabular} 
	\end{center}
\end{table}

\begin{figure}[tbhp]\label{figexp}
	\centering
	\subfloat[$n=10$, $p=5$  and  $M=5$]{\label{exp:a}\includegraphics[width=10cm]{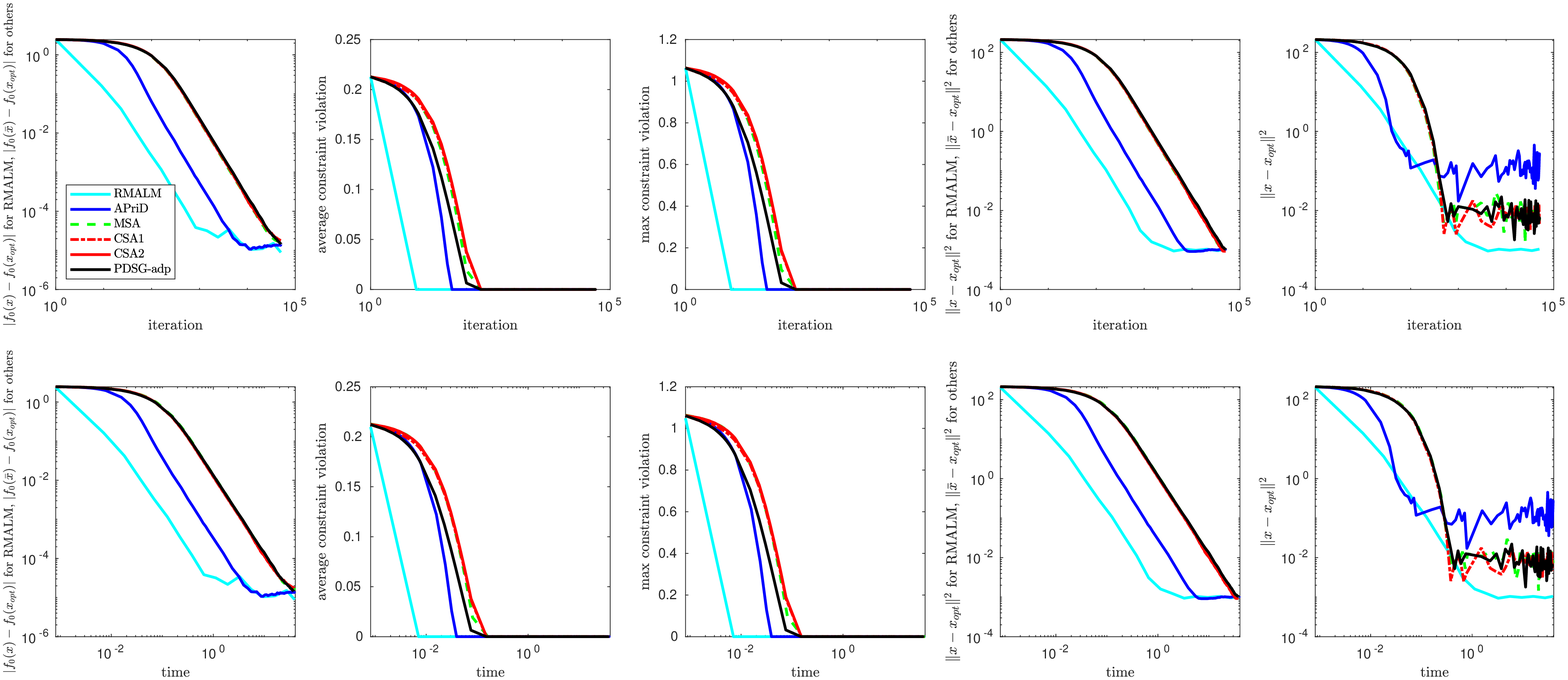}}
	
	\subfloat[$n=10$, $p=5$  and  $M=10000$]{\label{exp:b}\includegraphics[width=10cm]{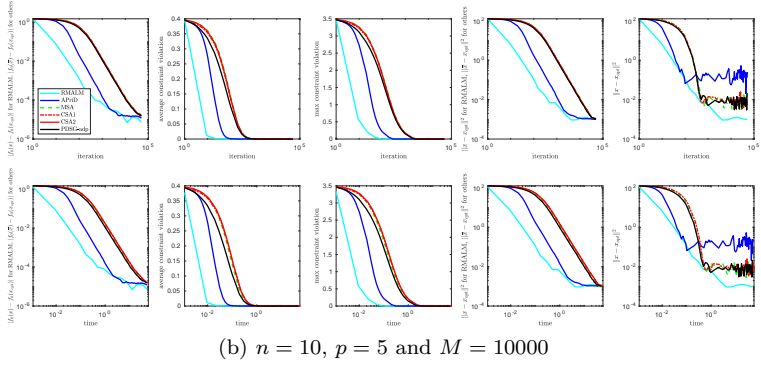}}
	
	\subfloat[$n=200$, $p=150$  and  $M=5$]{\label{exp:c}\includegraphics[width=10cm]{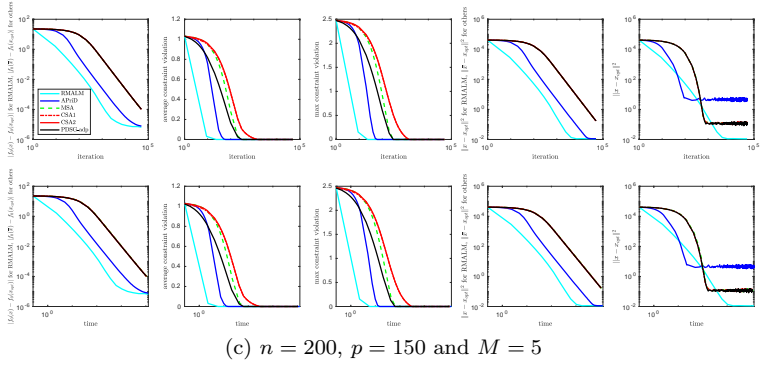}}
	
	\subfloat[$n=200$, $p=150$ and $M=10000$]{\label{exp:d}\includegraphics[width=10cm]{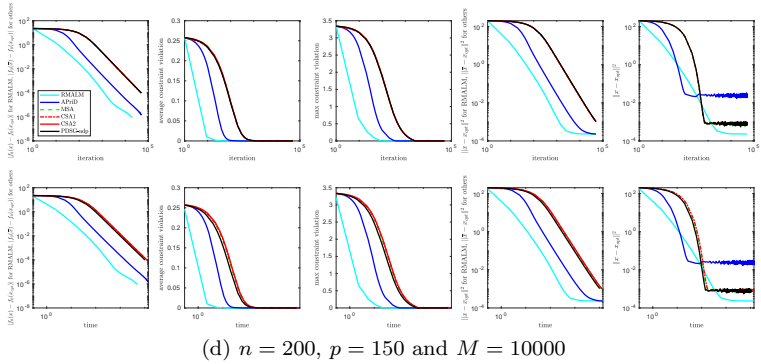}}
	
	\caption{In each subplot, the objective error (Left1), averaged constraint violation (Left2), maximum constraint violation (Middle), $x^k$ error for RMALM and $\bar{x}^k$ error for others (Right2), $x^k$ error (Right1) by five methods on solving QCQP instances of \cref{qcqp}. Rows 1 is with respect to iteration; rows 2 is with respect to time (in seconds).}
\end{figure}

%\begin{figure}[tbhp]
%	\centering
%	\ContinuedFloat
%
%	
%	\caption{In each subplot, the objective error (Left1), averaged constraint violation (Left2), maximum constraint violation (Middle), $x^k$ error for RMALM and $\bar{x}^k$ error for others (Right2), $x^k$ error (Right1) by five methods on solving QCQP instances of \cref{qcqp}. Rows 1 is with respect to iteration; rows 2 is with respect to time (in seconds).}
%\end{figure}

\subsection{QCQP with finite-sum objective}
In this subsection, we test the algorithms on the QCQP with a finite-sum objective and some constraints:

\begin{equation} \label{qcqp2}
	\begin{aligned}
		&\min _{{x} \in X}\ f(x)= \frac{1}{2N} \sum_{i=1}^{N}\left\|H_i x-c_i\right\|^{2},\\
		&\text { s.t. } h_j(x)=\frac{1}{2} x^{\top} Q_{j} x+a_{j}^{\top} x \leq b_{j},\ j=1, \ldots, M.
	\end{aligned}
\end{equation} 
Here $X=[-10,10]^{n}$. $H_i, c_i\ (i=1,\ldots,N)$ are independently generated from the same distribution as $\xi_H,\xi_c$ in Section \ref{qcqptest} and $Q_{j},a_j,b_j\ (j=1,\ldots,M))$ are generated in the same way as in Section \ref{qcqptest}. In this experiment, we test on  QCQP instances of size $(n, p) = (10, 5)$ and $(200, 150)$ and $M=5$ and $10000$. In both instances, we set $N = 10^4$, $batchsize=50$ and run $5 \times 10^{4}$ iterations.

Optimal solution $x_{opt}$ is obtained in the same way as in Section \ref{qcqptest}. We record the errors and constraints violations in \cref{finite} and the running time for the algorithms in \cref{finitetime}. From the results, we can also notice the better performance of the RMALM for different data sizes and number of constraints.

\begin{figure}[tbhp]\label{finite}
	\centering
	\subfloat[$n=10$, $p=5$ and $M=5$]{\label{fig:a}\includegraphics[width=10cm]{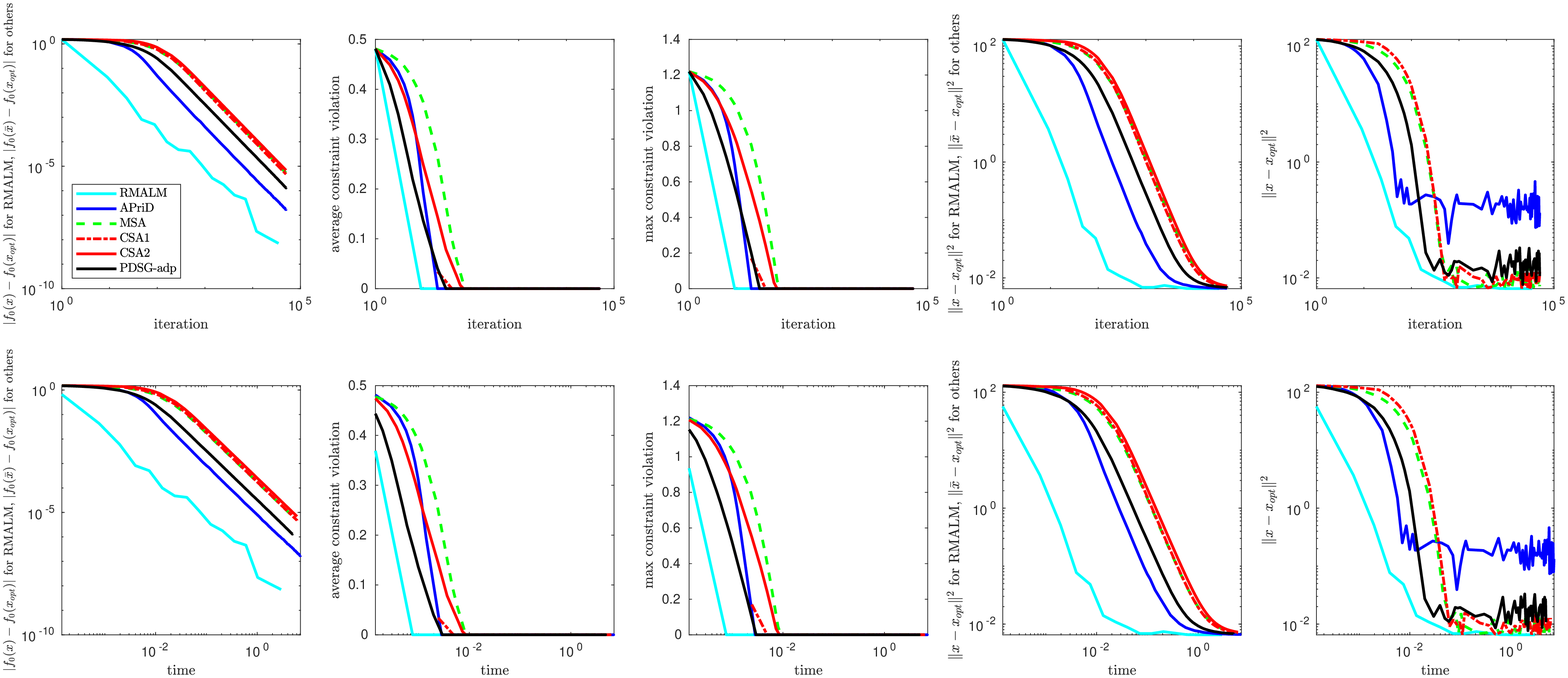}}
	
	\subfloat[$n=10$, $p=5$ and $M=10000$]{\label{fig:b}\includegraphics[width=10cm]{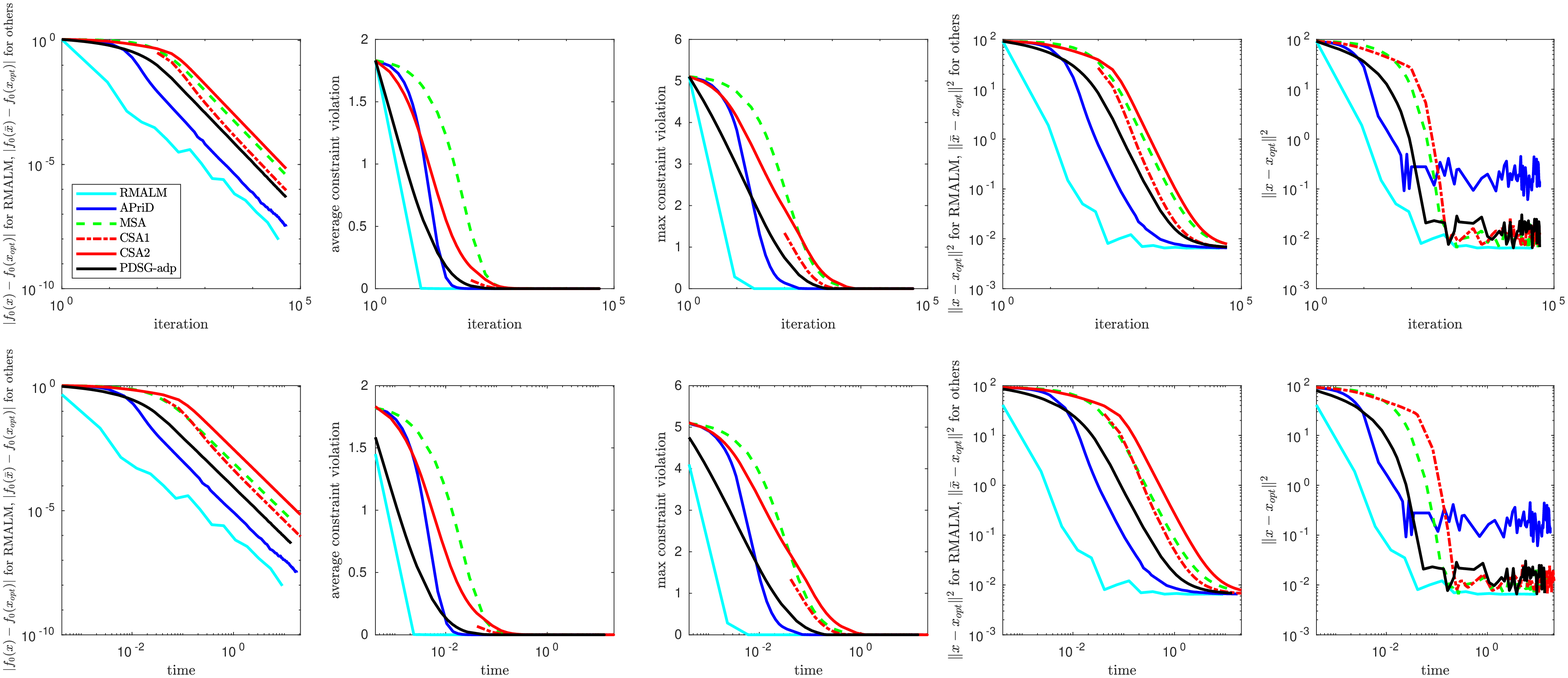}}
	
	\subfloat[$n=200$, $p=150$ and $M=5$]{\label{fig:c}\includegraphics[width=10cm]{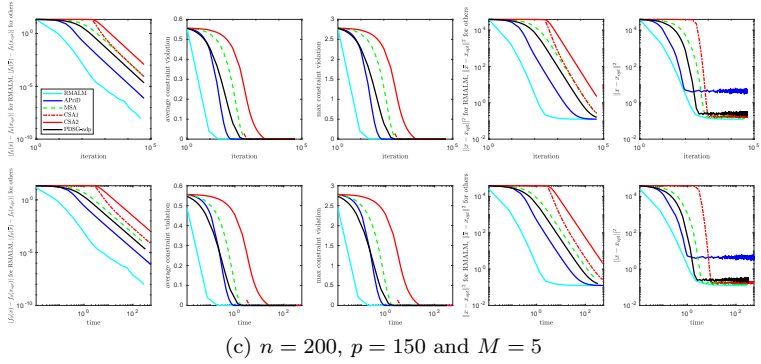}}
	
	\subfloat[$n=200$, $p=150$ and $M=10000$]{\label{fig:d}\includegraphics[width=10cm]{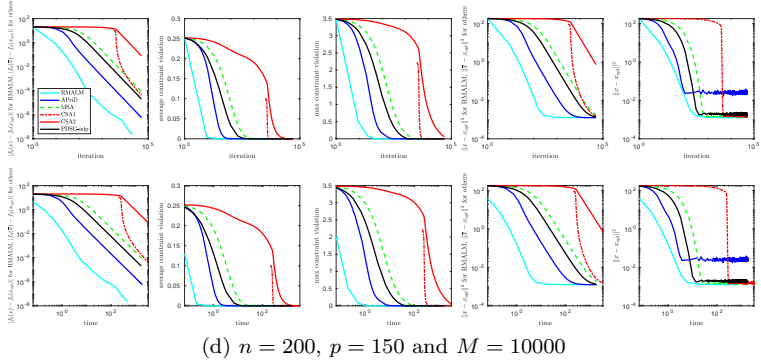}}
	
	\caption{In each subplot, the objective error (Left1), averaged constraint violation (Left2), maximum constraint violation (Middle), $x^k$ error for RMALM and $\bar{x}^k$ error for others (Right2), $x^k$ error (Right1) by five methods on solving QCQP instances of \cref{qcqp2}. Rows 1 is with respect to iteration; rows 2 is with respect to time (in seconds).}
\end{figure}

\begin{table}[tbhp]\label{finitetime}
	\footnotesize
	\captionsetup{position=top} %<- Needed for using subtables created with the subfig package
	\caption{ Running time (in seconds) for QCQP \cref{qcqp2}.}\label{tab1}
	\begin{center}
		\begin{tabular}{|c|c|c|c|c|} \hline
			$(n,p)$  & \multicolumn{2}{c|}{$(10,5)$} & \multicolumn{2}{c|}{$(200,150)$} \\ \hline
			number of constraints & {$M=5$} & {$M=10000$}  & {$M=5$} & {$M=10000$}\\ \hline
			RMALM & {\bf 3.0} & {\bf 9.1} & {\bf 235.5} & {\bf 793.2} \\ \hline
			ApriD &   7.2 & 17.5 & 482.2&1798.2 \\ \hline
			MSA & 5.3 & 14.8  & 259.2 & 1558.0 \\ \hline
			CSA & 5.9 & 20.7 & 461.6 & 2579.9 \\ \hline
			PDSG\_adp & 5.0 & 13.6 & 251.2 & 1492.0  \\ \hline
		\end{tabular}
	\end{center}
\end{table}

\subsection{Two-stage stochastic program}
We perform the RMALM on a specific example of the two-stage stochastic program introduced in Section \ref{intro}. Given by \cite{guigues2021inexact}, the program is:
\begin{equation}\label{one}
	\begin{aligned}	 
		\min\limits_{x_{1} \in \mathbb{R}^{n}} &\ c^{T} x_{1}+\mathbb{E}\left(\mathfrak{Q}\left(x_{1}, \xi\right)\right) \\
		\mbox{s.t. } &\ \left\|x_{1}-x_{0}\right\|_{2} \leq 1,
	\end{aligned}
\end{equation} 
where cost-to-go function $\mathfrak{Q}\left(x_{1}, \xi\right)$ has nonlinear objective and constraint coupling functions and is given by
\begin{equation} \label{two}
	\begin{aligned}
		\mathfrak{Q}\left(x_{1}, \xi\right):=
		\min \limits_{x_{2} \in \mathbb{R}^{n}}&\ \frac{1}{2}\left(\begin{array}{c}
			x_{1} \\
			x_{2}
		\end{array}\right)^{T}\left(\xi \xi^{T}+\lambda I_{2 n}\right)\left(\begin{array}{c}
			x_{1} \\
			x_{2}
		\end{array}\right)+\xi^{T}\left(\begin{array}{l}
			x_{1} \\
			x_{2}
		\end{array}\right) \\ 
		\mbox{s.t.}\ &\  \frac{1}{2}\left\|x_{2}-y_{0}\right\|_{2}^{2}+\frac{1}{2}\left\|x_{1}-x_{0}\right\|_{2}^{2}-\frac{R^{2}}{2} \leq 0.
	\end{aligned}
\end{equation}
For both problems, $\xi\in \mathbb{R}^{2 n}$ is generated from the Gaussian distribution and $\lambda>0$. The components of $\xi$ are independent with means and standard deviations randomly generated in intervals $[5,25]$ and $[5,15]$. We consider  {two} instances of these problem with $n=5,30$ and a large sample of size $N = 20,000$ of $\xi$. We fix $\lambda=2$ while the components of $c$ are generated randomly in interval $[1,3]$.

For specific N, we transform the problem to a single quadratic program
\begin{equation}\label{single1}
	\begin{aligned}
		\min\limits_{x_1 ,y_1,\ldots,y_N } &\ c^{T} x_{1} + \frac{1}{N}\sum_{i=1}^{N} \frac{1}{2}\left(\begin{array}{c}
			x_{1} \\
			y_{i}
		\end{array}\right)^{T}\left(\xi \xi^{T}+\lambda I_{2 n}\right)\left(\begin{array}{c}
			x_{1} \\
			y_{i}
		\end{array}\right)+\xi^{T}\left(\begin{array}{l}
			x_{1} \\
			y_{i}
		\end{array}\right)\\
		\mbox{s.t.}\quad\ &\left\|x_{1}-x_{0}\right\|_{2} \leq 1,\  	\frac{1}{2}\left\|y_{i}-y_{0}\right\|_{2}^{2}+\frac{1}{2}\left\|x_{1}-x_{0}\right\|_{2}^{2}-\frac{R^{2}}{2} \leq 0,\quad i=1,2,\ldots,N,
	\end{aligned}
\end{equation}
where $y_i$ represents the second stage decision corresponding to $\xi_i$. For problem \cref{single1} we take $R=5$ and $x_{0}(i)=y_{0}(i)=10, i=1, \ldots, n$. In both instances, we set batch size $=100$ and run $5 \times 10^{4}$ iterations. 

We record the running time and optimal value in \cref{stagetabel}. All constraint violations in five algorithms reach $0$. It shows that the RMALM can get optimal value in a faster time.

\begin{table}[tbhp]\label{stagetabel}
	\footnotesize
	\caption{CPU time in seconds, approximate optimal value of instances about problems \cref{one}-\cref{two} (for $n=5$ or $30$ and $N=20000$)}
	\begin{center}
		\subfloat[$n=5$]{
			\begin{tabular}{|c|c|c|} \hline
				method &  time(s) & optimal value   \\ \hline
				RMALM & {\bf 16.8} & {\bf 170.78}   \\ \hline
				ApriD & 108.6 & 171.06   \\ \hline	
				MSA & 49.6 & 179.86    \\ \hline	
				CSA1 & 53.7 & 171.12    \\ \hline	
				CSA2 & 53.7& 171.65    \\ \hline	
				PDSG\_adp & 52.8 & 180.69   \\ \hline
		\end{tabular}}   $\quad$
		\subfloat[$n=30$]{
			\begin{tabular}{|c|c|c|} \hline
				method &  time(s) & optimal value   \\ \hline
				RMALM & {\bf 52.0} & {\bf 1941.41}    \\ \hline
				ApriD & 505.1  & 1944.32   \\ \hline	
				MSA & 137.1  &    2061.07 \\ \hline	
				CSA1 & 132.5  & 1969.41  \\ \hline	
				CSA2 & 132.5   &  1974.48  \\ \hline	
				PDSG\_adp &150.1   &  2077.19 \\ \hline
		\end{tabular} } 
	\end{center}
\end{table}

\subsection{Stochastic portfolio optimization}
We perform the RMALM on the portfolio optimization problem involving Conditional Value at Risk (CVaR) shown as \cref{cvar3} and \cref{cvar5} on the finite dataset
\begin{equation} \label{cvar5}
	\begin{aligned}
		\min _{a, x \in X, y}\ &a+\frac{1}{(1-p)N} \sum_{i=1}^{N} y_{i}\\
		\mbox{s.t.}\quad  &y_{i} \geq -\xi_{i}^Tx -a, \quad y_{i} \geq 0,\quad  i=1, \ldots, N,
	\end{aligned}
\end{equation}
where $x$ represents the portfolio, $\xi_i$ denotes the rate of return corresponding to the investment at the $i$-th sample,  {$p\in(0,1)$ is a safety (reliability) level chosen by users, $a$ is a threshold of loss,} $N$ represents the number of samples. Together with  {the feasible set} \cref{set}, we can rewrite \cref{cvar5} as:
\begin{equation} \label{cvar4}
	\begin{aligned}
		\min _{a, x, y}\ &a+\frac{1}{(1-p)N} \sum_{i=1}^{N} y_{i}\\
		\mbox{s.t.}\ &y_{i} \geq -x^T\xi_{i}-a, \quad y_{i} \geq 0,\quad i=1, \ldots, N,\\
		& -m^{T}x  \leq-R,\quad \sum_{j=1}^{n}x_j = 1,\quad 0\leq x_j \leq 1,
	\end{aligned}
\end{equation}
where  {$m := \mathbb{E}(\xi)$ is the average return, $R$ encodes a minimum desired return.} Without loss of generality, we set the desired return as the average return of overall assets in the training set, i.e.,  {$R := \mathrm{mean}(m)$}. 

We test on five different real portfolio datasets: Dow Jones industrial average (DJIA, with 30 stocks for 507 days), Standard \& Poor’s 500 (SP500, with 25 stocks for 1276 days), Toronto stock exchange (TSE, with 88 stocks for 1258 days), New York stock exchange (NYSE, with 36 stocks for 5651 days) which are also used in \cite{borodin2003can,yurtsever2016stochastic}; and one dataset Fama and French (FF100, 100 portfolios formed on size and book-to-market, 25,251 days from July 1926 to May 2022) which is commonly used in financial literature, e.g., \cite{brodie2009sparse,nakagawa2021taming}. We complete the missing data in FF100 using the K-nearest neighbor method with Euclidean distance.

In both instances, we set batch size $=100$ and run $5 \times 10^{4}$ iterations. Then, for the $p$-values 0.95, we calculated the $p$-CVaR of the optimal portfolio $x^*$ from the formulas in \cref{cvar4}, obtaining the results in \cref{cvartable}, which records the running time, approximate optimal value, and averaged constraint violation for different datasets. The table shows that the RMALM algorithm can optimize the objective function value in a faster time for the same magnitude of constraint violation in these five real datasets.
\begin{table}[tbhp]\label{cvartable}
	\footnotesize
	\begin{center}
		\caption{CPU time in seconds, approximate optimal value and averaged constraint violation of problem \cref{cvar4} (for DJIA, SP500, TSE, NYSE and FF100)}
		\subfloat[DJIA, $(N,n)=(507,30)$]{
			\begin{tabular}{|c|c|c|c|} \hline
				method &  time(s) & optimal value & averaged constraint violation \\ \hline
				RMALM & {\bf 1.16} & {\bf -0.9747} & {\bf 3.3e-6} \\ \hline
				ApriD & 2.53 & -0.9114 & 4.2e-6 \\ \hline	
				MSA &  2.13 &-0.9057&  6.2e-6\\ \hline	
				CSA1 & 3.41 & -0.8457 & 8.3e-5  \\ \hline	%8.3
				CSA2 & 3.41& -0.6794&    2.5e-4  \\ \hline	  %2.5
				PDSG\_adp & 2.10 & -0.9730& 7.4e-6 \\ \hline
		\end{tabular}}   
		
		\subfloat[SP500, $(N,n)=(1276,25)$]{
			\begin{tabular}{|c|c|c|c|} \hline
				method &  time(s) & optimal value & averaged constraint violation \\ \hline
				RMALM & {\bf 2.24} & {\bf -0.9499} & {\bf 1.1e-6}\\ \hline
				ApriD & 5.20 & -0.9283 &  1.4e-5\\ \hline	   %
				MSA &  3.69 & -0.8853&  1.2e-5\\ \hline	
				CSA1 & 5.99 & -0.8588 & 9.2e-6\\ \hline	
				CSA2 & 5.99& -0.6048&   5.0e-5 \\ \hline	
				PDSG\_adp & 3.67 & -0.9453& {\bf 1.1e-6}\\ \hline   %
		\end{tabular}}   
		
		\subfloat[TSE, $(N,n)=(1258,88)$]{
			\begin{tabular}{|c|c|c|c|} \hline
				method &  time(s) & optimal value & averaged constraint violation \\ \hline
				RMALM & {\bf 2.85} & {\bf -0.9650} & {\bf 7.1e-6}\\ \hline
				ApriD & 7.08 & -0.8777 & 8.9e-6 \\ \hline	
				MSA &  5.59 &-0.8633 &  9.1e-6\\ \hline	
				CSA1 & 10.76 &  -0.8763 & 1.3e-5  \\ \hline	% 1.3
				CSA2 & 10.76 &  -0.6300&    1.9e-4 \\ \hline	 %1.94
				PDSG\_adp & 5.83 & -0.9590 &  7.3e-6\\ \hline   %
		\end{tabular}}  		
		
		\subfloat[NYSE, $(N,n)=(5651,36)$]{
			\begin{tabular}{|c|c|c|c|} \hline
				method &  time(s) & optimal value & averaged constraint violation \\ \hline
				RMALM & {\bf 3.98} & {\bf -1.0024} & {\bf 7.0e-6} \\ \hline
				ApriD & 14.36 & -0.9229 & 8.2e-6 \\ \hline	
				MSA &  6.61 &-0.5617 &  7.1e-6\\ \hline	
				CSA1 & 9.45 & -0.8684 & 8.4e-6\\ \hline	
				CSA2 & 9.45 &  -0.5896&   2.1e-5 \\ \hline  % 2.1
				PDSG\_adp & 6.93 & -0.9992 &  7.3e-6 \\ \hline
		\end{tabular}}  
		
		\subfloat[FF100, $(N,n)=(25251,100)$]{
			\begin{tabular}{|c|c|c|c|} \hline
				method &  time(s) & optimal value & averaged constraint violation \\ \hline
				RMALM & {\bf 10.60} &  {\bf 5.1800} & {\bf 4.1e-6}\\ \hline    %
				ApriD &  42.43 & 5.9690 & 4.4e-6 \\ \hline	
				MSA &  19.99 & 15.4154 &  {\bf 4.1e-6}\\ \hline	%
				CSA1 &   25.85 &  24.3529 & 4.8e-6 \\ \hline	
				CSA2 &    25.85&    20.5283&  7.9e-6  \\ \hline	%  
				PDSG\_adp & 18.87 & 5.3721 &  5.4e-6\\ \hline   %
		\end{tabular}} 
	\end{center}
\end{table}

%\begin{table}[tbhp]
%	\footnotesize
%	\ContinuedFloat
%	\begin{center}
	%		\caption{CPU time in seconds, approximate optimal value and average constraint violation of problem \cref{cvar4} (for DJIA, SP500, TSE, NYSE and FF100)}
	%		\subfloat[FF100, $(N,n)=(25251,100)$]{
		%			\begin{tabular}{|c|c|c|c|} \hline
			%				method &  time(s) & optimal value & averaged constraint violation \\ \hline
			%				RMALM & 10.60 &  5.1800 & $4\times 10^{-6}$ \\ \hline    %
			%				ApriD &  42.43 & 5.9690 & $4\times 10^{-6}$ \\ \hline	
			%				MSA &  19.99 & 15.4154 &  $4\times 10^{-6}$\\ \hline	%
			%				CSA1 &   25.85 &  24.3529 & $5\times 10^{-6}$  \\ \hline	
			%				CSA2 &    25.85&    20.5283&  $8\times 10^{-2}$  \\ \hline	%  
			%				PDSG\_adp & 18.87 & 5.3721 &  $5\times 10^{-6}$ \\ \hline   %
			%		\end{tabular}} 
	
	%	\end{center}%\end{table}
%\end{table}

%\begin{table}[htbp]
%\footnotesize
%\caption{Example table.}\label{tab:foo}
%\begin{center}
%  \begin{tabular}{|c|c|c|} \hline
	%   Species & \bf Mean & \bf Std.~Dev. \\ \hline
	%    1 & 3.4 & 1.2 \\
	%    2 & 5.4 & 0.6 \\ 
	%    3 & 7.4 & 2.4 \\ 
	%    4 & 9.4 & 1.8 \\ \hline
	%  \end{tabular}
%\end{center}
%\end{table}

\section{Conclusions}
\label{conclusion}

We present a hybrid method of stochastic approximation technique and augmented Lagrangian method for constrained stochastic convex optimization. The complexity is shown to be comparable with the existing related stochastic methods. Numerical experiments also demonstrate superiority in comparison with the first-order stochastic methods. Thus, both theoretical and numerical results suggest that the proposed algorithm is efficient for solving stochastic convex optimization with hard projection constraints. Our algorithm can also be extended to solve online constrained problems with determinate constraints. However, there are still several important issues to be studied. Firstly, our algorithm currently considers stochastic convex optimization with determinate constraints. Secondly, the convergence analysis of RMALM is guaranteed by the strong concavity of the dual essential objective function, and we will further consider weakening this assumption. Another interesting topic is how to use the techniques in this paper to deal with nonconvex constrained stochastic optimization, such as training neural networks with constraints.  

%\appendix

%\section*{Acknowledgments}
%We would like to acknowledge the assistance of volunteers in putting together this example manuscript and supplement.

\bibliographystyle{siamplain}
\bibliography{RMALM_main}

}
\end{document}